\documentclass[11pt]{amsart}
\usepackage[latin1]{inputenc}

\usepackage{epsfig}
\usepackage{mathrsfs}
\usepackage{color}
\usepackage[british,english]{babel}
\usepackage{amsthm}
\usepackage{amsmath}
\usepackage{amsfonts}
\usepackage{amssymb}
\usepackage{graphicx}
\numberwithin{equation}{section}

\date{\today}

\makeindex \makeatletter
\def\captionof#1#2{{\def\@captype{#1}#2}}
\makeatother

\newtheorem{lemma}{Lemma}[section]
\newtheorem{remark}{\bf Remark}
\newtheorem{theorem}{\bf Theorem}[section]
\def\bgneqnn{\begin{equation}}
\def\endeqnn{\end{equation}}
\def\bgneqy{\begin{eqnarray}}
\def\endeqy{\end{eqnarray}}
\def\bgneqy*{\begin{eqnarray*}}
\def\endeqy*{\end{eqnarray*}}

\def\text{\mbox}

\def\bgneqy*{\begin{eqnarray*}}
\def\endeqy*{\end{eqnarray*}}

\def\qed{\hfill$\blacksquare$\par\bigskip}

\newcounter{tablegroup}
\newcounter{subtable}[tablegroup]

\newcommand{\handletables}

\def\DST{\displaystyle}
\begin{document}
\title[General Decay in some Timoshenko-type systems...]{General Decay in some Timoshenko-type systems with thermoelasticity
second sound}
\author[Ayadi]{Mohamed Ali Ayadi}
\address{UR ANALYSE NON-LIN\'EAIRE ET G\'EOMETRIE, UR13ES32, Department of Mathematics, Faculty of Sciences of Tunis, University of Tunis El-Manar, 2092 El Manar II, Tunisia}
\email{ayadi.dali23@gmail.com}
\author[Bchatnia]{Ahmed Bchatnia}
\address{UR ANALYSE NON-LIN\'EAIRE ET G\'EOMETRIE, UR13ES32, Department of Mathematics, Faculty of Sciences of Tunis, University of Tunis El-Manar, 2092 El Manar II, Tunisia}
 \email{ahmed.bchatnia@fst.rnu.tn}
 \author[Hamouda]{Makram Hamouda}
\address{Institute for Scientific Computing and Applied Mathematics, Indiana University, 831 E. 3rd St., Rawles Hall, Bloomington IN 47405, United States}
\email{mahamoud@indiana.edu}
\author[Messaoudi]{Salim Messaoudi}
\address{Departement of Mathematics and Statistics, King Fahd University of Petroleum and Minerals, Dhahran 31261, Saudi Arabia}
\email{messaoud@kfupm.edu.sa}

\begin{abstract}
In this article, we consider a vibrating nonlinear Timoshenko system with thermoelasticity with second sound. We discuss the well-posedness and the regularity of Timoshenko solution using the semi-group theory. Moreover, we etablish an explicit and general decay results for a wide class of relaxating functions which depend on a stability number $\mu$.
\end{abstract}


\maketitle

\tableofcontents

\section{Introduction and setting of the problem}
\selectlanguage{english}

Beams represent the most common structural component found in civil
and mechanical structures. Because of their ubiquity they are extensively
studied, from an analytical viewpoint, in mechanics of materials. A
widely used mathematical model for describing the transverse vibrations of
beams is based on Timoshenko beam theory \textbf{TBT} (or thick beam theory)
developed by Timoshenko in the 1920's. The \textbf{TBT} accounts for both
the effect of rotational inertia and shear deformation that occur within a
beam as it vibrates. These factors are neglected when applied to
Euler-Bernoulli beam theory \textbf{EBT} (or thin beam theory), which is
appropriate for beams with small cross-sectional dimensions compared to the
length. In fact, a fundamental assumption in \textbf{EBT} is that cross
sections remain plane and normal to the deformed longitudinal axis
throughout deformation, while in \textbf{TBT} cross sections remain plane
but do not remain normal to the deformed longitudinal axis as the shear
deformation is taken into account. The cross section rotation from the
reference to the current configuration is denoted by $\varphi $ in both
models. In the \textbf{EB} model, this is the same as the rotation of the
longitudinal axis. In the Timoshenko model, the difference is used as
measure of mean shear distortion.

In 1921, Timoshenko \cite{Timoshenko S} gave the following system of coupled hyperbolic
equations
\begin{equation}
\left\lbrace
\begin{array}{l}
\rho u_{tt}=(K(u_{x}-\varphi ))_{x},\mbox{ }\,\hspace{2.52cm}\mbox{ in }(0,L)\times
\mathrm{I\hskip-2ptR}_{+}, \\
I_{\rho }\varphi _{tt}=(EI\varphi _{x})_{x}+K(u_{t}-\varphi ),\hspace{1.1cm}\mbox{ in
}\,(0,L)\times \mathrm{I\hskip-2ptR}_{+},
\end{array}\right. \label{Timo1}
\end{equation}
together with boundary conditions of the form
\[
EI\varphi _{x}|_{x=0}^{x=L}=0,\quad (u_{x}-\varphi )|_{x=0}^{x=L}=0,
\]
as a simple model describing the transverse vibrations of a beam. Here $t$
denotes the time variable and $x$ is the space variable along the beam of
length $L,$ in its equilibrium configuration, $u$ is the transverse
displacement of the beam and $\varphi $ is the rotation angle of the
filament of the beam. The coefficients $\rho ,I_{\rho },E,I$ and $K$ are
respectively the density (the mass per unit length), the polar moment of
inertia of a cross section, Young's modulus of elasticity, the moment of
inertia of a cross section, and the shear modulus.

System (\ref{Timo1}), with the above given boundary conditions, is conservative and
the natural energy of the beam, given by
\[
\mathcal{E}(t)=\frac{1}{2}\int_{0}^{L}\left( \rho |u_{t}|^{2}+I_{\rho
}|\varphi _{t}|^{2}+EI|\varphi _{x}|^{2}+K|u_{x}-\varphi |^{2}\right) dx,
\]
remains constant in time.

Vibration has long been known for its capacity of disturbance, discomfort,
damage and destruction. Since a long time, many researchers have been investigating ways to control this phenomenon. However, with the development of
control theory for partial differential equations over the last few decades,
it is not surprising that the issue of stability and controllability of
Timoshenko-type systems has received a great attention of many
mathematicians. One effective method for vibration control is passive
damping. Damping is most beneficial when used to reduce the amplitude of
dynamic instabilities, or resonances, in a structure.\newline

Damping is the conversion of mechanical energy of a structure into thermal
energy. A structure subject to oscillatory deformation contains a
combination of kinetic and potential energy.

 A damping effect may be caused by applying the beam to internal or boundary
frictional mechanisms. Depending of the nature of the beam's material, a damping effect may be rotating beam. For Viscoelastic materials with long memory, some beams are characterized by possessing both
viscous and elastic behavior. As a result of this behavior, some of the
energy stored in a viscoelastic system is recovered upon removal of the
load, and the remainder is dissipated in the form of heat.

Kim and Renardy \cite{KimRenardy} considered (\ref{Timo1}) together with two boundary controls of
the form
\begin{eqnarray*}
K\varphi (L,t)-K\frac{\partial u}{\partial x}(L,t) &=&\alpha \frac{\partial u%
}{\partial t}(L,t)\quad \forall t\geq 0, \\
EI\frac{\partial \varphi }{\partial x}(L,t) &=&-\beta \frac{\partial \varphi
}{\partial t}(L,t)\quad \forall t\geq 0,
\end{eqnarray*}
and used the multiplier techniques to establish an exponential decay result
for the natural energy of (\ref{Timo1}). They also provided numerical estimates to the
eigenvalues of the operator associated with system (\ref{Timo1}). An analogous result
was also established by Feng \textit{et al.} \cite{FengShi}, where the stabilization of
vibrations in a Timoshenko system was studied. Raposo \textit{et al}. \cite{RaposoFerreira}
studied (\ref{Timo1}) with homogeneous Dirichlet boundary conditions and two linear
frictional dampings. Precisely, they looked into the following system
\begin{equation}
\left\lbrace
\begin{array}{l}
\rho _{1}u_{tt}-K(u_{x}-\varphi )+u_{t}=0,\hspace{2.38cm}\mbox{ in }\,(0,L)\times
\mathrm{I\hskip-2ptR}_{+}, \\
\rho _{2}\varphi _{tt}-b\varphi _{xx}+K(u_{x}-\varphi )+\varphi _{t}=0,\hspace{1cm}
\mbox{ in }\,(0,L)\times \mathrm{I\hskip-2ptR}_{+}, \\
u(0,t)=u(L,t)=\varphi (0,t)=\varphi (L,t)=0,\quad t>0
\end{array}\right. \label{Timo2}
\end{equation}
and proved that the energy associated with (\ref{Timo2}) decays exponentially.
Soufyane and Wehbe \cite{SoufyaneWehbeUnifo} showed that it is possible to stabilize uniformly
(\ref{Timo1}) by using a unique locally distributed feedback. They considered
\begin{equation}
\left\lbrace
\begin{array}{l}
\rho u_{tt}=(K(u_{x}-\varphi ))_{x},\hspace{3.7cm}\mbox{ in }\,(0,L)\times \mathrm{I\hskip-2ptR}_{+},
\\
I_{\rho }\varphi _{tt}=(EI\varphi _{x})_{x}+K(u_{x}-\varphi )-b\varphi
_{t},\hspace{0.9cm}\mbox{ in }\,(0,L)\times \mathrm{I\hskip-2ptR}_{+} ,\\
u(0,t)=u(L,t)=\varphi (0,t)=\varphi (L,t)=0,\quad t>0,
\end{array}\right. \label{Timo3}
\end{equation}
where $b$ is a positive and continuous function, which satisfies
\[
b(x)\geq b_{0}>0,\quad \forall \mbox{ }x \in
[a_{0},a_{1}]\subset [0,L].
\]
In fact, they proved that the uniform stability of (\ref{Timo3}) holds if and only if
the wave speeds are equal $\left( \frac{K}{\rho }=\frac{EI}{I_{\rho }}%
\right) ;$ otherwise only the asymptotic stability has been proved. Rivera
and Racke \cite{MunozRiveraRackeTimo} obtained a similar result in a work, where the damping
function $b=b(x)$ is allowed to change sign. They also in treated \cite{MunozRiveraRackeGlobal} a nonlinear Timoshenko-type system of the form
\[\left\lbrace
\begin{array}{l}
\rho _{1}\varphi _{tt}-\sigma _{1}(\varphi _{x},\psi )_{x}=0, \\
\rho _{2}\psi _{tt}-\chi (\psi _{x})_{x}+\sigma _{2}(\varphi _{x},\psi
)+d\psi _{t}=0,
\end{array}\right.
\]
in a one-dimensional bounded domain. The dissipation here is through
frictional damping which is only in the equation for the rotation angle. The
authors gave an alternative proof for a sufficient and necessary condition
for exponential stability in the linear case and then proved a polynomial
stability in general. Moreover, they investigated the global existence of
small smooth solutions and exponential stability in the nonlinear case.

Shi and Feng \cite{ShiFeng} used the frequency multiplier method to investigate a
nonuniform Timoshenko beam and showed that, under some locally distributed
controls, the vibration of the beam decays exponentially. The nonuniform
Timoshenko beam has also been studied by Ammar-Khodja \textit{et al.} \cite{AmmarKerbal}
and a similar result to that in \cite{ShiFeng} has been established.

Ammar-Khodja \textit{et al.} \cite{AmmarBenabdallah} considered a linear Timoshenko-type system
with memory of the form
\begin{equation}
\left\lbrace
\begin{array}{l}
\rho _{1}\varphi _{tt}-K(\varphi _{x}+\psi )_{x}=0, \\
\rho _{2}\psi _{tt}-b\psi _{xx}+\int_{0}^{t}g(t-s)\psi _{xx}(s)ds+K(\varphi
_{x}+\psi )=0, \\
\varphi (x,0)=\varphi _{0}(x),\mbox{ }\varphi _{t}(x,0)=\varphi _{1}(x), \\
\psi (x,0)=\psi _{0}(x),\mbox{ }\psi _{t}(x,0)=\psi _{1}(x) \\
\varphi (0,t)=\varphi (1,t)=\psi (0,t)=\psi (1,t)=0,
\end{array}\right. \label{Timo4}
\end{equation}
in$\,(0,L)\times \mathrm{I\hskip-2ptR}_{+},$ and proved, using the multiplier techniques,
that the system is uniformly stable if and only if the wave speeds are equal
$\left( \frac{K}{\rho _{1}}=\frac{b}{\rho _{2}}\right) $ and $g$ decays
uniformly. More precisely, they proved an exponential decay if $g$ decays in an
exponential rate and polynomially if $g$ decays in a polynomial rate. They
also required some extra technical conditions on both $g^{\prime }$ and $%
g^{\prime \prime }$ to obtain their results. This result has been later
improved by Messaoudi and Mustafa \cite{MessaoudiMustafa4} and Guesmia and Messaoudi \cite{GuesmiaMessaoudi1}, where
the technical conditions on $g^{\prime \prime }$ have been removed and those
on $g^{^{\prime }}$ have been weakened$.$ Also, Guesmia and Messaoudi \cite{GuesmiaMessaoudi2}
considered the following system
\begin{equation}
\left\{
\begin{array}{l}
\rho _{1}\varphi _{tt}-K(\varphi _{x}+\psi )_{x}=0, \\
\rho _{2}\psi _{tt}-\kappa \psi _{xx}+\int_{0}^{t}g(t-\tau )(a(x)\psi
_{x}(\tau ))_{x}d\tau +K(\varphi _{x}+\psi )
+b(x)h(\psi _{t})=0 ,\\
\varphi (x,0)=\varphi _{0}(x),\mbox{ }\varphi _{t}(x,0)=\varphi _{1}(x), \\
\psi (x,0)=\psi _{0}(x),\mbox{ }\psi _{t}(x,0)=\psi _{1}(x), \\
\varphi (0,t)=\varphi (1,t)=\psi (0,t)=\psi (1,t)=0,
\end{array}
\right.\label{Timo5}
\end{equation}
in$\,(0,1)\times \mathrm{I\hskip-2ptR}_{+}.$ They proved under similar conditions on the relaxation function $g$, which are similar to those in \cite{FernandezRacke}, and by assuming that
\[
a(x)+b(x)\geq \rho >0, \ \forall x\in (0,1),
\]
an exponential stability for $g$ decaying exponentially and $h$ linear, and
polynomial stability when $g$ decays polynomially and $h$ is nonlinear.

Concerning stabilization via heat effect, Rivera and Racke \cite{MunozRiveraRackeMildy} investigated
the following system
\[\left\lbrace
\begin{array}{l}
\rho _{1}\varphi _{tt}-\sigma (\varphi _{x},\psi )_{x}=0,\hspace{3.2cm}\mbox{ in }\,(0,L)\times \mathrm{I\hskip-2ptR}_{+}, \\
\rho _{2}\psi _{tt}-b\psi _{xx}+K(\varphi _{x}+\psi )+\gamma \theta _{x}=0,\hspace{0.5cm}\mbox{ in }\,(0,L)\times \mathrm{I\hskip-2ptR}_{+}
,\\
\rho _{3}\theta _{t}-K\theta _{xx}+\gamma \psi _{xt}=0,\hspace{2.8cm}\mbox{ in },(0,L)\times \mathrm{I\hskip-2ptR}_{+},
\end{array}\right.
\]
where $\varphi ,\psi ,\theta $ are functions of $(x,t)$ model the transverse
displacement of the beam, the rotation angle of the filament, and the
difference temperature respectively. Under appropriate conditions on $\sigma
,\rho _{i},b,K,\gamma ,$ they proved several exponential decay results for
the linearized system and non exponential stability result for the case of
different wave speeds.

Concerning Timoshenko systems of thermoelasticity with second sound,
Messaoudi \textit{et al.} \cite{MessaoudiPokojovy} studied
\[\left\lbrace
\begin{array}{l}
\rho _{1}\varphi _{tt}-\sigma (\varphi _{x},\psi )_{x}+\mu \varphi
_{t}=0,\hspace{2.52cm}\mbox{ in }\,(0,L)\times  \mathrm{I\hskip-2ptR}_{+} ,\\
\rho _{2}\psi _{tt}-b\psi _{xx}+k(\varphi _{x}+\psi )+\beta \theta
_{x}=0,\hspace{1.1cm}\mbox{ in }\, (0,L)\times  \mathrm{I\hskip-2ptR}_{+}, \\
\rho _{3}\theta _{t}+\gamma q_{x}+\delta \psi _{tx}=0,\hspace{3.7cm}\mbox{ in }\,
(0,L)\times  \mathrm{I\hskip-2ptR}_{+} ,\\
\tau _{0}q_{t}+q+\kappa \theta _{x}=0,\hspace{4.3cm}\mbox{ in }\, (0,L)\times  \mathrm{I\hskip-2ptR}_{+},
\end{array}\right.
\]
where $\varphi =\varphi (x,t)$ is the displacement vector, $\psi =\psi (x,t)$
is the rotation angle of the filament, $\theta =\theta (x,t)$ is the
temperature difference, $q=q(x,t)$ is the heat flux vector, $\rho _{1}$, $%
\rho _{2}$, $\rho _{3}$, $b$, $k$, $\gamma $, $\delta $, $\kappa $, $\mu $, $%
\tau _{0}$ are positive constants. The nonlinear function $\sigma $ is
assumed to be sufficiently smooth and satisfy
\[
\sigma _{\varphi _{x}}(0,0)=\sigma _{\psi }(0,0)=k,
\]
and
\[
\sigma _{\varphi _{x}\varphi _{x}}(0,0)=\sigma _{\varphi _{x}\psi
}(0,0)=\sigma _{\psi \psi }=0.
\]
Several exponential decay results for both linear and nonlinear cases have
been established in the presence of the extra frictional damping $\mu
\varphi _{t}$.

Fern\'{a}ndez Sare and Racke \cite{FernandezRacke} considered
\begin{equation}
\left\lbrace
\begin{array}{l}
\rho _{1}\varphi _{tt}-k(\varphi _{x}+\psi )_{x}=0,\hspace{3.4cm}\mbox{ in }\,
(0,L)\times  \mathrm{I\hskip-2ptR}_{+}, \\
\rho _{2}\psi _{tt}-b\psi _{xx}+k(\varphi _{x}+\psi )+\delta \theta
_{x}=0,\hspace{1.11cm}\mbox{ in }\, (0,L)\times  \mathrm{I\hskip-2ptR}_{+}, \\
\rho _{3}\theta _{t}+\gamma q_{x}+\delta \psi _{tx}=0,\hspace{3.7cm}\mbox{ in }\,
(0,L)\times  \mathrm{I\hskip-2ptR}_{+}, \\
\tau q_{t}+q+\kappa \theta _{x}=0,\hspace{4.3cm}\mbox{ in }\,  (0,L)\times  \mathrm{I\hskip-2ptR}_{+},
\end{array}\right. \label{Timo6}
\end{equation}
and showed that, in the absence of the extra frictional damping ($\mu =0$),
the coupling via Cattaneo's law causes loss of the exponential decay
usually obtained in the case of coupling via Fourier's law \cite{MunozRiveraRackeMildy}. This surprising property holds even for systems with history of the form
\begin{equation}
\left\lbrace
\begin{array}{l}
\rho _{1}\varphi _{tt}-k(\varphi _{x}+\psi )_{x}=0,\hspace{7cm}\mbox{ in }\,
(0,L)\times  \mathrm{I\hskip-2ptR}_{+}, \\
\rho _{2}\psi _{tt}-b\psi _{xx}+k(\varphi _{x}+\psi )+\int_{0}^{+\infty
}g(s)\psi _{xx}(.,t-s)ds+\delta \theta _{x}=0,\hspace{0.1cm}\mbox{ in }\,
(0,L)\times  \mathrm{I\hskip-2ptR}_{+}, \\
\rho _{3}\theta _{t}+\gamma q_{x}+\delta \psi _{tx}=0, \hspace{7.33cm}\mbox{ in }\,
(0,L)\times  \mathrm{I\hskip-2ptR}_{+},\\
\tau q_{t}+q+\kappa \theta _{x}=0,\hspace{7.95cm}\mbox{ in }\,
(0,L)\times  \mathrm{I\hskip-2ptR}_{+},
\end{array}\right. \label{Timo7}
\end{equation}
 Precisely, it has been shown that both systems (\ref{Timo6}) and (\ref{Timo7}) are no longer
exponentially stable even for equal-wave speeds $\left( \frac{k}{\rho _{1}}=%
\frac{b}{\rho _{2}}\right) .$ However, no other rate of decay has been
discussed.\\
Very recently, Santos et al. \cite{SantosAlmeida} considered (\ref{Timo6}) and introduced a new
stability number
\[
\mu =\left( \tau -\frac{\rho _{1}}{k\rho _{3}}\right) \left( \frac{\rho _{2}}{b}-
\frac{\rho _{1}}{k}\right) -\frac{\rho _{1}\delta ^{2}\rho _{1}}{kb\rho _{3}},
\]
and used the semi-group method to obtain exponential decay result for $\mu =0
$ and a polynomial decay for $\mu \neq 0.$

The boundary feedback of memory type has also been used by Santos \cite{Santosdecay}. He
considered a Timoshenko system and showed that the presence of two feedbacks
of memory type at a portion of the boundary stabilizes the system uniformly.
He also obtained the rate of decay of the energy, which is exactly the rate
of decay of the relaxation functions. This last result has been improved and
generalized by Messaoudi and Soufyane \cite{MessaoudiSoufyane}. For more results concerning
well-posedness and controllability of Timoshenko systems, we refer the
reader to \cite{MessaoudiMustafa2,MessaoudiMustafa3}, \cite{MessaoudiSaid}, \cite{MunozRiveraRackeStabi}, \cite{ShiHou} and \cite{ShiFengYan,ShiFengExpo}.

In this paper we consider the following Timoshenko system:
\begin{equation}
\left\{
\begin{array}{l}
\rho _{1}\varphi _{tt}-k(\varphi _{x}+\psi )_{x}=0,\hspace{6.2cm}\textnormal{in }%
(0,1)\times \mathrm{I\hskip-2ptR}_{+}, \\
\rho _{2}\psi _{tt}-b \psi _{xx}+k(\varphi _{x}+\psi )
+\delta \theta _{x}+\alpha (t)h(\psi _{t})=0,\hspace{1.8cm} \textnormal{in }(0,1)\times
\mathrm{I\hskip-2ptR}_{+}, \\
\rho _{3}\theta _{t}+q_{x}+\delta \psi _{xt}=0,\hspace{6.73cm} \textnormal{in }(0,1)\times
\mathrm{I\hskip-2ptR}_{+}, \\
\tau q_{t}+\beta q+\theta _{x}=0,\hspace{7.2cm} \textnormal{in }(0,1)\times \mathrm{I\hskip%
-2ptR}_{+}, \\
\varphi _{x}(0,t)=\varphi _{x}(1,t)=\psi (0,t)=\psi (1,t)=q(0,t)=q(1,t)=0,%
\hspace{0.5cm}\forall\mbox{ }t\geq 0 ,\\
\varphi (x,0)=\varphi _{0}(x),\mbox{ }\varphi _{t}(x,0)=\varphi _{1}(x),\hspace{4.8cm}
\forall\mbox{ }x\in (0,1) ,\\
\psi (x,0)=\psi _{0}(x),\mbox{ }\psi _{t}(x,0)=\psi _{1}(x),\hspace{4.8cm}\forall\mbox{ }x\in
(0,1), \\
\theta (x,0)=\theta _{0}(x),\mbox{ }q(x,0)=q_{0}(x),\hspace{5.2cm}\forall\mbox{ }x\in (0,1),
\end{array}
\right.
\label{1}
\end{equation}
where, $\rho _{1}$, $%
\rho _{2}$, $\rho _{3}$, $b$, $k$, $\delta $, $\beta $ are positive constants, $\varphi =\varphi (x,t)$ is the displacement vector, $\psi =\psi (x,t)$
is the rotation angle of the filament, $\theta =\theta (x,t)$ is the
temperature difference and $q=q(x,t)$ is the heat flux vector.
Also, $ \alpha $ and $h$ are two functions to be fixed later.\\

Using $(\ref{1})_{1},\ (\ref{1})_{3}$ and the boundary conditions $\eqref{1}_{5}$, we have
$$\DST \frac{d^2}{dt^2}\int_{0}^{1}\varphi (x,t)dx=0 \mbox{ and } \frac{d}{dt}\int^{1}_{0} \theta (x,t)dx=0.$$
Consequently, we obtain
$$\int^{1}_{0} \varphi(x,t)dx=\left( \int^{1}_{0} \varphi _{1}(x)dx\right) t+\int^{1}_{0} \varphi_{0} (x)dx \mbox{ and } \int^{1}_{0} \theta (x,t)dx=\int^{1}_{0} \theta_{0} (x)dx.$$
If we set
$$\bar{\varphi}(x,t)=\varphi(x,t)-\left( \left( \int^{1}_{0} \varphi _{1}(x)dx\right) t+\int^{1}_{0} \varphi_{0} (x)dx\right), $$
and
$$\bar{\theta}(x,t)= \theta (x,t)-\int^{1}_{0} \theta_{0} (x)dx,$$
\\
then $(\bar{\varphi },\psi,\bar{\theta },q)$ satisfy also the system (\ref{1}), and we have
$$\int^{1}_{0} \bar{\varphi} (x,t)dx=0\hspace{1cm}\text{and}\hspace{1cm}\int^{1}_{0} \bar{\theta} (x,t)dx=0.$$
From now on, we use the new variables $(\bar{\varphi},\psi,\bar{\theta},q)$, but we denote them by $(\varphi,\psi,\theta,q)$, for simplicity. \\

The article is organized as follows. First, in Section \ref{section well}, we use the semi-group theory to prove the existence and uniqueness of solutions of system (\ref{1}). Next, in Section \ref{section stability}, we study the asymptotic behavior of the energy of solutions of system (\ref{1}) using the multiplier method. For that purpose, we assume some hypotheses on $\alpha$ and $h$. The optimal exponential and polynomial decay rate estimates can be obtained in some special cases with explicit nonlinear terms.\\
\section{Well-posedness and regularity}\label{section well}
In this section, we discuss the well-posedness of the problem (\ref{1}), using the semi-group theory. We consider the following hypotheses on $\alpha$ and $h$:\\
\vspace{0.3cm}
$(A_{1})$ : $\alpha$ : $\mathbb{R}_{+}\rightarrow \mathbb{R}_{+}$ is differentiable and decreasing.\\
\vspace{0.3cm}
$(A_{2})$ : $h : \mathbb{R}\rightarrow \mathbb{R}$ is a locally Lipschitz function satisfying  $h(0)=0 $.\\
We introduce the Hilbert space:
$$L_\star^2 (0,1)=\{v\in L^2 (0,1) \ :\ \DST \int_{0}^{1}v(s)ds=0 \},$$
$$H_\star^1 (0,1)= H^1 (0,1)\cap L_\star ^2(0,1),$$
$$H_\star^2 (0,1)=\{v\in H^2 (0,1) \ :\ v_x(0)=v_x(1)=0 \}.$$
The energy associated with the system (\ref {1}) is defined by:
$$\DST E(\varphi,\psi ,\theta ,q )(t)=\frac{1}{2}\int_{0}^{1} (\rho_{1}\varphi_{t}^{2}+\rho_{2}\psi _{t}^{2}+b\psi _{x}^{2}+k(\varphi_{x}+\psi )^{2}+\rho_{3}\theta^{2}+\tau q^{2})dx.$$
Let
$$\DST \ H=H_{\star}^{1}(0,1)\times L^{2}_{\star}(0,1)\times
H^{1}_{0}(0,1)\times L^{2}(0,1)\times L^{2}_{\star}(0,1)\times L^{2}(0,1),$$
be the Hilbert space endowed with the inner product defined, for \\$U=(
u_{1},u_{2},u_{3},u_{4},u_{5},u_{6})^{t}\in H,\ V=(
v_{1},v_{2},v_{3},v_{4},v_{5},v_{6}%
)^{t}\in H$, by
$$\DST \left\langle U,V\right\rangle _{H}= \rho_{1}\left\langle
u_{2},v_{2}\right\rangle _{L^{2}(0,1)}+\rho_{2}\left\langle
u_{4},v_{4}\right\rangle _{L^{2}(0,1)}
+k\left\langle u_{1x}+u_{3},v_{1x}+v_{3}\right\rangle _{L^{2}(0,1)}$$
$$+b\left\langle
u_{3x},v_{3x}\right\rangle _{L^{2}(0,1)}+\rho_{3}\left\langle
u_{5},v_{5}\right\rangle _{L^{2}(0,1)}+\tau \left\langle
u_{6},v_{6}\right\rangle _{L^{2}(0,1)}.$$
\vspace{0.2cm}\\
For $\Phi=(\varphi,u,\psi ,v,\theta,q)^{t}$ and $\Phi_{0}=(\varphi_{0},\varphi_{1},\psi_{0} ,\psi_{1},\theta_{0},q_{0})^{t}$, where $u=\varphi_{t}$ and $v=\psi _{t}$, (\ref{1}) is equivalent to the abstract first order Cauchy problem 
 \begin{equation}
  \left\{
\begin{array}{c}
\DST \frac{d }{dt}\Phi(t)+(A+B)\Phi(t)=0,\hspace{1cm}\forall \ t\in \mathbb{R}_{+}, \\
\Phi(0)=\Phi_{0},%
\end{array}%
\right.\label{firstorder}
 \end{equation}
 where $A :D(A)\subset H\longrightarrow H $ is the linear operator defined by
 \begin{equation}
A\Phi=\left(
\begin{array}{cccc}
-u \\ -\frac{k}{\rho_{1}}\varphi _{xx} -\frac{k}{\rho_{1}}\psi_{x}  \\ -v \\  -\frac{b}{\rho_{2}}\psi_{xx}+\frac{k}{\rho_{2}} (\varphi _{x}+\psi)+\frac{\delta}{\rho_{2}}\theta_{x} \\  \frac{1}{\rho_{3}}q_{x}+\frac{\delta}{\rho_{3}}v_{x} \\ \frac{\beta}{\tau}q + \frac{1}{\tau}\theta_{x}
\end{array}%
\right ),\label{a*fi}
\end{equation}
and $B :D(B)\subset H\longrightarrow H $ is the nonlinear operator defined by
$$\DST
B\Phi=
\left(
\begin{array}{cccc}
0 \\ 0 \\ 0 \\ \alpha(t) h(v) \\ 0 \\ 0
\end{array}%
\right ).$$
The domain of the operator $A$ is given by
$\DST D (A)=\lbrace \Phi \in H\ ;\ A\Phi \in H \rbrace$
and endowed with the graph norm
$$\Vert\Phi \Vert_{D(A)} =\Vert \Phi\Vert_{H} +\Vert A\Phi\Vert_{H},$$
can be characterized by
\[\DST
D(A)=(H^{2}_{*}( 0,1)\cap H_{*}^{1}( 0,1)) \times
H^{1}_{*}( 0,1) \times (H^{2}( 0,1)\cap
H_{0}^{1}( 0,1)) $$
$$\times H^{1}_{0}( 0,1)\times H^{1}_{*}( 0,1) \times H^{1}_{0}( 0,1).
\]
The domain of the operator B is given by $D(B)=\lbrace  \Phi \in H\ ;\ B\Phi \in H\rbrace = H.$ \\
We first state and prove the following lemmas which will be useful to deduce the well-posedness result.

\begin{lemma}\label{lemmonoton}
For $\Phi \in D(A)$, we have $(A\Phi,\Phi )_{H}\geq0$.
\end{lemma}
\begin{proof}For any $\Phi=(\varphi,u,\psi ,v,\theta,q)^{t}\in D(A)$, we have
$$(A\Phi,\Phi)_{H}=k\int_{0}^{1}-(u_{x}+v)(\varphi_{x}+\psi))dx+\int_{0}^{1}(-k\varphi_{xx}-k\psi_{x})udx +b \int_{0}^{1} -v_{x}\psi_{x}dx$$$$+ \int_{0}^{1}(-b\psi_{xx}+k(\varphi_{x}+\psi)+\delta \theta_{x})vdx + \int_{0}^{1}(q_{x}+\delta v_{x})\theta dx +\int_{0}^{1}(\beta q+ \theta_{x})qdx. $$

Using integration by parts and the boundary conditions in (\ref{1}), we obtain \\
$$(A\Phi,\Phi)_{H}=\beta\int_{0}^{1}q^{2}dx\geq 0.$$
This ends the proof of the lemma.
\end{proof}
\begin{lemma}\label{lemsurjective}
 $I+A$ is a surjective operator.
\end{lemma}
\begin{proof}
For any $W=(w_{1},w_{2},w_{3},w_{4},w_{5},w_{6})\in H$, we prove that there exists $V=(v_{1},v_{2},v_{3},v_{4},v_{5},v_{6})\in D(A)$ satisfying $$(I+A)V=W.$$
That is,
\begin{equation}\label{surjective}
\left\lbrace
\begin{array}{l}
-v_{2}+v_{1}=w_{1},\\
 -kv_{1xx}-kv_{3x}+\rho_{1}v_{1}=\rho_{1}(w_{1}+w_{2}),\\
-v_{4}+v_{3}=w_{3},\\
 -bv_{3xx}+k(v_{1x}+v_{3})+\delta v_{5x}+\rho_{2}v_{4}=\rho_{2}w_{4},\hspace{1cm}\\ v_{6x}+\delta v_{4x}+\rho_{3}v_{5}=\rho_{3}w_{5},\\
(\beta +\tau)v_{6}+v_{5x}=\tau w_{6}.
\end{array}\right.
\end{equation}
Then $(\ref{surjective})_{1},(\ref{surjective})_{3} $ and $(\ref{surjective})_{5}$ yield
\begin{equation}
v_{2}=v_{1}-w_{1}\in H_{*}^{1}(0,1),\label{surj1}
\end{equation}
\begin{equation}
v_{4}=v_{3}-w_{3}\in H_{0}^{1}(0,1),\label{surj3}
\end{equation}
$$v_{6x}=\rho_{3}w_{5} + \delta w_{3x}-\delta v_{3x}-\rho_{3}v_{5}.$$
By integration over $(0,x)$ and using $v_{6}(0)=w_{3}(0)=v_{3}(0)=0$, we obtain
\begin{equation}
 v_{6}=\rho_{3}\int _{0}^{x}w_{5}ds + \delta w_{3}-\delta v_{3}-\rho_{3}\int_{0}^{x}v_{5}ds.\label{surj5}
 \end{equation}
We substitute (\ref{surj5}) into $(\ref{surjective})_{6}$ and we get

$$v_{5x}+(\beta +\tau)\left[ \rho_{3}\int _{0}^{x}w_{5}ds + \delta w_{3}-\delta v_{3}-\rho_{3}\int_{0}^{x}v_{5}ds \right] =\tau w_{6}.
$$
Hence, we deduce that
\begin{equation}
\DST -v_{5x}+(\beta+\tau)\delta v_{3}+\rho_{3}(\beta+\tau)\int_{0}^{x}v_{5}ds  =(\beta+\tau)\delta w_{3}+(\beta +\tau)\rho_{3}\int _{0}^{x}w_{5}ds -\tau w_{6}.
\label{surj6}
\end{equation}
Again, we substitute (\ref{surj6}) into $(\ref{surjective})_{4}$, we get
$$\DST -bv_{3xx}+kv_{1x}+kv_{3}+\delta \left[ (\beta+\tau)\delta v_{3}+\rho_{3}(\beta+\tau)\int_{0}^{x}v_{5}ds -(\beta+\tau)\delta w_{3}\right.  $$$$\left. -(\beta +\tau)\rho_{3}\int _{0}^{x}w_{5}ds -\tau w_{6}\right] +\rho_{2}v_{3}=\rho_{2}(w_{3}+w_{4}),$$
and we infer that
\begin{eqnarray}
\DST -bv_{3xx}+kv_{1x}+kv_{3}+\delta^{2}  (\beta+\tau)\delta v_{3}+\rho_{3}\delta (\beta+\tau)\int_{0}^{x}v_{5}ds +\rho_{2}v_{3} =(\beta+\tau)\delta^{2} w_{3}\nonumber\\
\DST +(\beta +\tau)\delta\rho_{3}\int _{0}^{x}w_{5}ds -\delta \tau w_{6} +\rho_{2}(w_{3}+w_{4}).\hspace{2cm}
\label{surj4}
\end{eqnarray}
By using (\ref{surj6}), (\ref{surj4}) and $(\ref{surjective})_{2}$, it can be shown that $v_{1}$, $v_{3}$ and $v_{5}$ satisfy

\begin{equation}\label{mult1}
\left\lbrace
\begin{array}{l}
\DST -kv_{1xx}-kv_{3x}+\rho_{1}v_{1}=h_{1}\in L^{2}_{*}(0,1),\\
 \DST -bv_{3xx}+kv_{1x}+kv_{3}+(\delta^{2}  (\beta+\tau) +\rho_{2})v_{3}+\rho_{3}\delta (\beta+\tau)\int_{0}^{x}v_{5}ds=h_{2}\in L^{2}(0,1),\\
\DST -\rho_{3}v_{5x}+\rho_{3}(\beta+\tau)\delta v_{3}+\rho_{3}^{2}(\beta+\tau)\int_{0}^{x}v_{5}=h_{3}\in L^{2}(0,1),
 \end{array}\right.
 \end{equation}
 where
\begin{equation}
\left\lbrace
\begin{array}{l}
h_{1}=\rho_{1}(w_{1}+w_{2}),\hspace{4cm}\nonumber\\
h_{2}=(\beta+\tau)\delta^{2} w_{3}+(\beta +\tau)\delta\rho_{3}\int _{0}^{x}w_{5}ds -\delta \tau w_{6} +\rho_{2}(w_{3}+w_{4}),\nonumber\\
 h_{3}=\rho_{3}(\beta+\tau)\delta w_{3}+(\beta +\tau)\rho_{3}^{2}\int _{0}^{x}w_{5}ds -\rho_{3}\tau w_{6}.\nonumber
\end{array}\right.
 \end{equation}
Let $u=(u_{1},u_{3},u_{5})$ and $v=(v_{1},v_{3},v_{5})$, a simple multiplication of $(\ref{mult1})_{1}$, $(\ref{mult1})_{2}$ and $(\ref{mult1})_{3}$, by $u_{1},u_{3} $ and $\DST \int_{0}^{x}u_{5}ds$ respectively, and integration over $(0,1)$ yield
\begin{eqnarray}
\hspace{1cm}&\bullet &-k\int_{0}^{1}v_{1xx}u_{1}dx-k\int_{0}^{1}v_{3x}u_{1}dx+\rho_{1}\int_{0}^{1}v_{1}u_{1}dx=\int_{0}^{1}h_{1}u_{1}dx,\label{fi}\\
\hspace{1cm}&\bullet & -b\int_{0}^{1}v_{3xx}u_{3}dx+k\int_{0}^{1}v_{1x}u_{3}dx+k\int_{0}^{1}v_{3}u_{3}dx+(\delta^{2}  (\beta+\tau) +\rho_{2})\int_{0}^{1}v_{3}u_{3}dx\nonumber\\
\hspace{1cm}&\ &\hspace{2cm}+\rho_{3}\delta (\beta+\tau)\int_{0}^{1}(\int_{0}^{x}v_{5}ds)u_{3}dx=\int_{0}^{1}h_{2}u_{3}dx,\nonumber\\
 &\bullet &-\rho_{3}\int_{0}^{1}v_{5x}(\int_{0}^{x}u_{5}ds)dx+\rho_{3}(\beta+\tau)\delta \int_{0}^{1}v_{3}(\int_{0}^{x}u_{5}ds)dx+\nonumber\\
& \ &\hspace{2cm} \rho_{3}^{2}(\beta+\tau)\int_{0}^{1}(\int_{0}^{x}v_{5}ds)(\int_{0}^{x}u_{5}ds)dx=\int_{0}^{1}h_{3}(\int_{0}^{x}u_{5}ds)dx.\nonumber
 \end{eqnarray}
 Using integration by parts and the boundary conditions yield

\begin{eqnarray}
&\bullet &k\int_{0}^{1}v_{1x}u_{1x}dx+k\int_{0}^{1}v_{3}u_{1x}dx+\rho_{1}\int_{0}^{1}v_{1}u_{1}dx=\int_{0}^{1}h_{1}u_{1}dx,\nonumber\\
&\bullet & b\int_{0}^{1}v_{3x}u_{3x}dx+k\int_{0}^{1}v_{1x}u_{3}dx+k\int_{0}^{1}v_{3}u_{3}dx+(\delta^{2}  (\beta+\tau) +\rho_{2})\int_{0}^{1}v_{3}u_{3}dx\nonumber\\
&\ &\hspace{2cm}+\rho_{3}\delta (\beta+\tau)\int_{0}^{1}(\int_{0}^{x}v_{5}ds)u_{3}dx=\int_{0}^{1}h_{2}u_{3}dx,\nonumber\\
 &\bullet &\rho_{3}\int_{0}^{1}v_{5x}u_{5}dx+\rho_{3}(\beta+\tau)\delta \int_{0}^{1}v_{3}(\int_{0}^{x}u_{5}ds)dx+\nonumber\\
& \ & \hspace{2cm}\rho_{3}^{2}(\beta+\tau)\int_{0}^{1}(\int_{0}^{x}v_{5}ds)(\int_{0}^{x}u_{5}ds)dx=\int_{0}^{1}h_{3}(\int_{0}^{x}u_{5}ds)dx.\nonumber
 \end{eqnarray}
The sum of the previous equations gives the following variational formulation
 \begin{equation}
   b(v,u)=l(u),\label{formulationvar}
   \end{equation}
 for all $u=(u_{1},u_{3},u_{5})\in  H_{*}^{1}(0,1)\times H_{0}^{1}(0,1)\times L^{2}_{*}(0,1) $, where b is defined by
$$b(v,u)=k\int_{0}^{1}(v_{1x}+v_{3})(u_{1x}+u_{3})dx+\rho_{1}\int_{0}^{1}v_{1}u_{1}dx+b\int_{0}^{1}v_{3x}u_{3x}dx$$$$+(\delta^{2}  (\beta+\tau) +\rho_{2})\int_{0}^{1}v_{3}u_{3}dx+\rho_{3}\delta (\beta+\tau)\int_{0}^{1}(\int_{0}^{x}v_{5}ds)u_{3}dx
 +\rho_{3}\int_{0}^{1}v_{5x}u_{5}dx$$$$+\rho_{3}(\beta+\tau)\delta \int_{0}^{1}v_{3}(\int_{0}^{x}u_{5}ds)dx+\rho_{3}^{2}(\beta+\tau)\int_{0}^{1}(\int_{0}^{x}v_{5}ds)(\int_{0}^{x}u_{5}ds)dx,$$
and $l$ is defined by
$$l(u)=\int_{0}^{1}h_{1}u_{1}dx
+\int_{0}^{1}h_{2}u_{3}dx
+\int_{0}^{1}h_{3}(\int_{0}^{x}u_{5}ds)dx.$$
We introduce the Hilbert space $\Lambda = H_{*}^{1}(0,1)\times H_{0}^{1}(0,1)\times L^{2} (0,1)$
 equipped with the norm
 $$\Vert v\Vert_{\Lambda}^{2}=\Vert v_{1x}+v_{3}\Vert^{2}_{2} + \Vert v_{1}\Vert^{2}_{2}  + \Vert v_{3x}\Vert ^{2}_{2} + \Vert v_{5}\Vert^{2}_{2}. $$
It is clear that b is a bilinear and continuous form on $\Lambda \times \Lambda$, and $l$ is a linear and continuous form on $\Lambda$. Furthermore, there exists a positive constant $c_{0}$ such that
$$b(v,v)= k\Vert  v_{1x}+v_{3}\Vert _{2}^{2}+\rho_{1}\Vert v_{1}\Vert_{2}^{2}+b\Vert v_{3x}\Vert_{2}^{2}+(\delta^{2}(\beta+\tau)+\rho_{2})\Vert v_{3}\Vert_{2}^{2}+\rho_{3}\Vert v_{5}\Vert_{2}^{2}$$
$$+2\rho_{3}(\beta+\tau)\delta \int_{0}^{1}v_{3}(\int_{0}^{x}v_{5}ds)dx+\rho_{3}^{2}(\beta+\tau)\int_{0}^{1}(\int_{0}^{x}v_{5}ds)^{2}dx$$
\hspace{2.5cm}$\geq c_{0}\Vert v\Vert_{\Lambda}^{2}.$
\\
which implies that b is coercive.\\ Therefore, using the Lax-Milgram theorem we conclude that the system (\ref{mult1}) has a unique solution $$(v_{1},v_{3},v_{5})\in (H_{*}^{1}(0,1)\times H_{0}^{1}(0,1)\times L_{*}^{2}(0,1)),$$
and we deduce from (\ref{surj1})-(\ref{surj5}) the existence of
$v_{2}\in  H_{*}^{1}(0,1),$ $v_{4}\in H_{0}^{1}(0,1),$ and
$v_{6}\in L_{*}^{2}(0,1))\subset L^{2}(0,1)).$ \\
Now, it remains to show that
$$v_{1}\in  H_{*}^{2}(0,1)\cap H_{*}^{1}(0,1),\ v_{3}\in  H^{2}(0,1)\cap H_{0}^{1}(0,1),\ \ v_{5}\in H_{*}^{1}(0,1)\ \text{ and }  v_{6}\in H_{0}^{1}(0,1).$$
From (\ref{mult1}), we have
$$-kv_{1xx}=kv_{3x}-\rho_{1}v_{1}+h_{1}\in L^{2}(0,1).$$
Consequently, it follows that $$v_{1}\in H^{2}(0,1)\cap H_{*}^{1}(0,1).$$
Moreover, (\ref{fi})  is also true for any $\varphi_{1} \in \mathcal{C}^{1}([0,1])$. Hence, we have
$$k\int_{0}^{1}v_{1x}\varphi_{1x}dx+k\int_{0}^{1}v_{3}\varphi_{1x}dx+\rho_{1}\int_{0}^{1}v_{1}\varphi_{1}dx=\int_{0}^{1}h_{1}\varphi_{1}dx,$$ for any $\varphi_1\in \mathcal{C}^{1}([0,1])$. Thus, using integration by parts we obtain
$$ v_{1x}(1)\varphi_{1}(1)- v_{1x}(0)\varphi_{1}(0) =0, \mbox{ for all } \varphi_1\in \mathcal{C}^{1}([0,1]).$$
Therefore, $v_{1x}(1)= v_{1x}(0)= 0,$ and we deduce that
$$v_{1}\in H^{2}_{*}(0,1)\cap H_{*}^{1}(0,1).$$
Now, we substitute $(\ref{surjective})_{6}$ into $(\ref{surjective})_{4}$, we get
$$bv_{3xx}=kv_{1x}+kv_{3}+\delta \tau w_{6}-\delta (\beta+\tau)v_{6} +\rho_{2}v_{3} -h_{2}\in L^{2}(0,1).$$
Consequently, it follows that $$v_{3}\in H^{2}(0,1)\cap H_{0}^{1}(0,1).$$
On the other hand, we get from $(\ref{surjective})_{6}$,
$$v_{5x}=\tau w_{6}-(\beta +\tau )v_{6}\in L^{2}(0,1),$$
and we deduce that $$v_{5}\in H^{1}(0,1)\cap L^{2}_{*}(0,1).$$
Similarly, from $(\ref{surjective})$ we have
$$v_{6x}=\rho_{3}w_{5} + \delta w_{3x}-\delta v_{3x}-\rho_{3}v_{5}\in L^{2}(0,1))\hspace{0.3cm}\text{ which implies } \hspace{0.3cm}v_{6}\in H_{0}^{1}(0,1),$$
as $v_{6}(0)=v_{6}(1)=0.$ \\
Finally, the operator $I+A$ is surjective.
\end{proof}

Using Lemmas \ref{lemmonoton} and \ref{lemsurjective}, we conclude that the operator $A+B$ is the infinitesimal generator of a non-linear contraction $C_{0}$-semi-group on the Hilbert space $H$.\\
Finally, by applying the semi-group theory to (\ref{firstorder}) (see \cite{Komornik,Pazy}), we easily get the following well-posedness result.
\begin{theorem}\label{thp}
Assume that $(A_{1})$ and $(A_{2})$ are satisfied, then for all initial data $$(\varphi_{0},\varphi_{1},\psi_{0},\psi_{1},\theta_{0},q_{0})\in (H^{2}_{\star}(0,1)\cap H_{\star}^{1}(0,1))\times H_{\star}^{1}(0,1)\times (H^{2}(0,1)\cap H_{0}^{1}(0,1))$$
$$\times H_{0}^{1}(0,1)\times H_{\star}^{1}(0,1)\times H_{0}^{1}(0,1),$$
the system (\ref {1}) has a unique solution $(\varphi,\psi,\theta,q)$ that verifies  \\
$\left. \right. \hspace{2cm}
(\varphi,\psi)\in C^{0}(\mathbb{R}_{+},(H^{2}_{\star}(0,1)\cap H_{\star}^{1}(0,1))\times(H^{2}(0,1)\cap H_{0}^{1}(0,1)))$\vspace{0.2cm} \\
$ \left. \right. \hspace{3cm}\cap \ C^{1}(\mathbb{R}_{+},H_{\star}^{1}(0,1)\times H_{0}^{1}(0,1))\cap \ C^{2}(\mathbb{R}_{+},L^{2}_{\star }(0,1)\times L^{2}(0,1)),$\\
and $$\hspace{1.8cm}(\theta,q)\in  \ C^{0}(\mathbb{R}_{+},H_{\star}^{1}(0,1)\times H_{0}^{1}(0,1))\cap \ C^{1}(\mathbb{R}_{+},L^{2}_{\star }(0,1)\times L^{2}(0,1)).$$
\end{theorem}

\section{Stability results}\label{section stability}
In this section, we state and prove a stability result for the nonlinear Timoshenko system \eqref{1}. For this purpose, we consider the following hypotheses:

$(A_{1})$ : $\alpha$ : $\mathbb{R}_{+}\rightarrow \mathbb{R}_{+}$ is a differentiable and decreasing function.

$(A_{2})^{*}$ : $h$ : $\mathbb{R}_{+}\rightarrow \mathbb{R}_{+}$ is a continuous non-decreasing function such that $h(0)=0$
and there exists a continuous strictly increasing odd function $h_{0}\in C([0,+\infty))$, continuously differentiable in a neighborhood of 0
and satisfying $h_{0}(0)=0$
\newline
$$ \left\{
\begin{array}{l}
h_{0}(\vert(s)\vert)\leq \vert h(s) \vert \leq h_{0}^{-1}(\vert(s)\vert),%
\hspace{2.7em} \text{for all}\hspace{1.2em} \vert s\vert\leq\varepsilon , \\
c_{1} \vert s \vert\leq \vert h(s)\vert \leq c_{2} \vert s\vert ,\hspace{%
6.4em}\text{ for all} \hspace{1.12em}\vert s\vert\geq\varepsilon .%
\end{array}
\right.$$ \\
where $c_{i} > 0$ for i = 1, 2. \\
Moreover, we define a function $H$ by
\begin{equation}
 H(x)=\sqrt{x}h_{0}(\sqrt{x})
 \end{equation}
Thanks to Assumption $(A_{2})^{*}$, $H$ is of class $C^{1}$ and is strictly convex on $(0,r^{2}]$, where $r > 0$ is
a sufficiently small number.

\begin{remark}\  \\
\begin{itemize}
\item  We denote by $c$ positive generic constant throughout this paper.
\item The hypothesis $A_{1}$ implies that $\alpha(t)\leq c$.
\end{itemize}
\end{remark}
We recall here the stability number defined by : $$\mu=\left[ (\tau -\frac{\rho_{1}}{k \rho_{3}})(\frac{\rho_{2}}{b}-\frac{\rho_{1}}{k})-\frac{\tau\delta^{2}\rho_{1}}{bk \rho_{3}}\right].$$
\subsection{The case $\mu = 0$}\label{sub31} \ \\
In this part, we state and prove the decay results which are not necessarily of exponential or polynomial types. For this purpose, we establish several lemmas. We recall that the energy associated with the system \eqref{1} is defined by
\begin{equation}
E(t):=\frac{1}{2}\int_{0}^{1}\left( \rho_{1} \varphi_{t}^{2}+\rho_{2}\psi _{t}^{2}+b \psi _{x}^{2}+k(\varphi_{x}+\psi) ^{2}+\rho_{3}\theta^{2}+\tau q^{2}\right) dx.\label{energy}
\end{equation}
Throughout the rest of this paper we assume that conditions $(A_{1})$ and $(A_{2})^{*}$ hold.
\begin{lemma}\label{lemmaenergie}
Let $(\varphi,\psi,\theta,q)$ be a solution of the system \eqref{1}. Then, the functional $E$ satisfies
\begin{equation}
E^{\prime
}(t)=-\beta \int^{1}_{0}q^{2}dx-\alpha(t)\int^{1}_{0}\psi_{t}h(\psi_{t}) dx \leq 0.
\label{energyde}
\end{equation}
\end{lemma}
\begin{proof}
By multiplying the first fourth equations in \eqref{1}, respectively, by $\varphi_{t}$, $\psi_{t}$, $\theta $ and $q$, using the integration by parts with respect to $x$ over $(0,1)$, the boundary conditions $\eqref{1}_{5}$ and the hypotheses $(A_{1})$ and $(A_{2})^{*}$, we obtain \eqref{energyde}.
\end{proof}
\begin{lemma}\label{lemmak1}
Let $(\varphi,\psi,\theta,q)$ be a solution of the system \eqref{1}. Then, the functional
 \begin{equation}
 K_{1}(t):=-\int^{1}_{0}(\rho_{1}\varphi\varphi_{t}+\rho_{2}\psi\psi_{t})dx, \label{k1}
 \end{equation}
verifies the following estimate
\begin{eqnarray}
K_{1}'(t)\leq &-&\rho_{1}\int^{1}_{0}\varphi_{t} ^{2}dx-\rho_{2}\int^{1}_{0}\psi_{t} ^{2}dx+c\int^{1}_{0}\psi_{x} ^{2}dx+k\int^{1}_{0}(\varphi_{x}+\psi ) ^{2}dx \label{k1'}\\
&+&\frac{\delta}{2}\int^{1}_{0}\theta ^{2}dx+\frac{1}{2}\int^{1}_{0}h ^{2}(\psi_{t})dx.\nonumber
\end{eqnarray}
\end{lemma}
\begin{proof}
 By differentiating \eqref{k1} and using the first and second equations of \eqref{1}, we get
\begin{eqnarray*}
K_{1}'(t)
=\!\!&-&\!\!\! \rho_{1}\int^{1}_{0}\varphi_{t}^{2}dx-\rho_{2}\int^{1}_{0}\psi_{t}^{2}dx-\int^{1}_{0}k(\varphi _{x}+\psi )_{x}\varphi dx-\int^{1}_{0}(b \psi _{xx}-k(\varphi _{x}+\psi )\\&-&\delta \theta _{x}-\alpha (t)h(\psi _{t}))\psi dx.\\
\end{eqnarray*}
Integrating by parts and using the boundary conditions $\eqref{1}_{5}$, we have
$$K_{1}'(t)=-\rho_{1}\int^{1}_{0}\varphi_{t}^{2}dx-\rho_{2}\int^{1}_{0}\psi_{t}^{2}dx+b\int^{1}_{0}\psi_{x}^{2}dx+\int^{1}_{0}k(\varphi _{x}+\psi )^{2}dx$$
$$-\delta \int^{1}_{0}\theta \psi_{x}dx+\int^{1}_{0}\alpha (t)h(\psi _{t})\psi dx.$$
Applying Young's inequality, we obtain \eqref{k1'}.
\end{proof}
\begin{lemma}\label{lemmak2}
Let $(\varphi,\psi,\theta,q)$ be a solution of the system \eqref{1}. Then, the functional
\begin{equation}
 K_{2}(t):=\rho_{2}\int^{1}_{0}\psi\psi_{t}dx -\rho_{2}\int^{1}_{0}\varphi_{t}wdx-\delta\tau \int^{1}_{0}\psi qdx\label{k2},
 \end{equation}
satisfies, for any $\varepsilon >0 $
\begin{eqnarray}
K_{2}'(t)\leq &-&\left( b-2c\varepsilon \right) \int_{0}^{1}  \psi _{x}^{2}dx+c(\int_{0}^{1}\psi _{t}^{2}dx+\int_{0}^{1}q^{2}dx+\int_{0}^{1}h^{2}(\psi_{t})dx) \label{k2'}\\
&+&\rho_{1}\varepsilon\int_{0}^{1}\varphi_{t}^{2}dx, \nonumber
\end{eqnarray}
where $w$ is the solution of the problem
\begin{equation}
\left\{
\begin{array}{l}
-w_{xx}=\psi_{x}, \\
w(0)=w(1)=0.%
\end{array}%
\right.  \label{w}
\end{equation}
\end{lemma}
\begin{proof}
 By differentiation of \eqref{k2} and the use of the first, second and fourth equations of \eqref{1}, we get
$$K_{2}'(t)=\rho_{2}\int_{0}^{1}\psi _{t}^{2}dx+b\int_{0}^{1}\psi _{xx}\psi dx-k\int_{0}^{1}(\varphi_{x}+\psi )\psi dx-\delta\int_{0}^{1}\theta_{x}\psi dx-\alpha (t)\int_{0}^{1}\psi h(\psi_{t})dx$$
$$+k\int_{0}^{1}(\varphi_{x}+\psi )_{x}wdx+\rho_{1}\int_{0}^{1}\varphi_{t}w_{t} dx-\tau\delta\int_{0}^{1}\psi _{t}qdx+\delta\beta \int_{0}^{1}\psi qdx+\delta \int_{0}^{1}\theta _{x}\psi dx.$$
Integrating by parts the last equality, using \eqref{w} and the boundary conditions $\eqref{1}_{5}$, we have
$$K_{2}'(t)=\rho_{2}\int_{0}^{1}\psi _{t}^{2}dx-b\int_{0}^{1}\psi _{x}^{2} dx-k\int_{0}^{1}\psi^{2} dx+k\int_{0}^{1}w_{x}^{2} dx-\alpha (t)\int_{0}^{1}\psi h(\psi_{t})dx$$
$$+\rho_{1}\int_{0}^{1}\varphi_{t}w_{t} dx-\tau\delta\int_{0}^{1}\psi _{t}qdx+\delta\beta \int_{0}^{1}\psi qdx.$$
By a simple calculation, we easily deduce that the function $w$ satisfies the following estimates
 \begin{equation}
\int_{0}^{1}w_{x}^{2} dx\leq \int_{0}^{1}\psi^{2} dx,\label{w1}
\end{equation}
\begin{equation}
\int_{0}^{1}w_{t}^{2} dx\leq c\int_{0}^{1}\psi_{t}^{2} dx. \label{w2}
\end{equation}
Thanks to Young's and Poincar\'e's inequalities and \eqref{w1}-\eqref{w2}, we conclude that
\begin{eqnarray}
K_{2}'(t)\leq & &\! \! \!\! \! \!\! \! \!\! \! \!\rho_{2}\int_{0}^{1}\! \! \!\psi _{t}^{2}dx-b\int_{0}^{1}\! \! \!\psi _{x}^{2} dx+\frac{\rho_{1}}{4\varepsilon}\int_{0}^{1}w_{t}^{2} dx+\rho_{1}\varepsilon\int_{0}^{1}\varphi_{t}^{2} dx \\
&+&\tau\delta\varepsilon\int_{0}^{1}\psi _{t}^{2}dx
+\frac{\tau\delta}{4\varepsilon} \int_{0}^{1}q^{2}dx+c_{p}\varepsilon \int_{0}^{1}\psi_{x}^{2}dx+\frac{(\delta\beta)^{2}}{4\varepsilon}\int_{0}^{1}q^{2}dx\nonumber\\
&+&\varepsilon c_{p}\int_{0}^{1}\psi_{x}^{2} dx+\frac{c^{2}}{4\varepsilon} \int_{0}^{1}h^{2}(\psi_{t})dx.\nonumber
\end{eqnarray}
Therefore, we obtain \eqref{k2'}.
\end{proof}

\begin{lemma}\label{lemmak3}
Let $(\varphi,\psi,\theta,q)$ be a solution of the system \eqref{1}. Then, the functional
\begin{equation}
 K_{3}(t):=-\tau\rho_{3}\int^{1}_{0}q(\int^{x}_{0}\theta(t,y)dy)dx, \label{k3}
 \end{equation}
satisfies
\begin{eqnarray}
K_{3}'(t)\leq -\frac{\rho_{3}}{2}\int_{0}^{1}\theta^{2}dx+c\left( \int_{0}^{1}q^{2}dx+\int_{0}^{1}\psi_{t}^{2}dx\right). \label{k3'}
\end{eqnarray}
\end{lemma}
\begin{proof}
By differentiation of \eqref{k3} and the use of the third and fourth equations of \eqref{1}, we get
\begin{eqnarray} K'_{3}(t)&=&\rho_{3}\beta\int^{1}_{0}q(\int^{x}_{0}\theta(t,y)dy)dx+\rho_{3}\int^{1}_{0}
\theta_{x}(\int^{x}_{0}\theta(t,y)dy)dx \nonumber \\
&&+\tau\int^{1}_{0}q(\int^{x}_{0}q_{x}(t,y)dy)dx+\tau\delta\int^{1}_{0}q(\int^{x}_{0}\psi_{tx}(t,y)dy)dx.\nonumber
\end{eqnarray}
 By integrating the above equality over $(0,1)$ and using the boundary conditions $\eqref{1}_{5}$ (note also that $\int^{1}_{0}\theta dx=0$), we have
  $$ K'_{3}(t)=\rho_{3}\beta\int^{1}_{0}q(\int^{x}_{0}\theta(t,y)dy)dx-\rho_{3}\int^{1}_{0}\theta^{2}dx+\tau\int^{1}_{0}q^{2}dx
+\tau\delta\int^{1}_{0}q\psi_{t}dx.$$
Applying again Young's inequality and the fact that
$$\int^{1}_{0}(\int^{x}_{0}\theta(t,y)dy)^{2}dx\leq c\int^{1}_{0}\theta^{2}dx,$$
we arrive at \eqref{k3'}.
\end{proof}
\begin{lemma}\label{lemmak4}
Let $(\varphi,\psi,\theta,q)$ be a solution of the system \eqref{1}. Then, the functional
\begin{eqnarray}
 K_{4}(t):=\! \! \! &&\! \! \! \! \! \! \frac{\tau\rho_{2}}{k}\int^{1}_{0}\psi_{t}(\varphi_{x}+\psi)dx+\frac{b\tau\rho_{1}}{k^{2}}\int^{1}_{0}\varphi_{t}\psi_{x}dx \label{k4}\\
 &-&\frac{b\tau\rho_{3}}{\delta k}(\frac{\rho_{2}}{b}-\frac{\rho_{1}}{k})\int^{1}_{0}\theta\varphi_{t}dx+\frac{b\tau}{\delta k}(\frac{\rho_{2}}{b}-\frac{\rho_{1}}{k})\int^{1}_{0}q(\varphi_{x}+\psi)dx,\nonumber
\end{eqnarray}
satisfies
\begin{eqnarray}
K_{4}'(t)\leq \! \! \! \! \!&-&\! \! \! \! \!(\tau -2\varepsilon_{1})\int_{0}^{1}\! \! \! \! (\varphi_{x}+\psi )^{2}dx+C\left(\int_{0}^{1}\! \! \! \! \!\psi_{t}^{2}dx+\int_{0}^{1}\! \! \!  q^{2}dx+\int_{0}^{1}\! \! \! \! h^{2}(\psi_{t})dx\right)\label{k4-prime}\\
&+&\frac{b\rho_{3}}{\delta \rho_{1}}\left[ (\tau -\frac{\rho_{1}}{k \rho_{3}})(\frac{\rho_{2}}{b}-\frac{\rho_{1}}{k})-\frac{\tau\delta^{2}\rho_{1}}{bk \rho_{3}}\right]\int^{1}_{0}\theta_{x}(\varphi_{x}+\psi )dx,\nonumber
\end{eqnarray}
with $C=2\max(\frac{\tau\rho_{2}}{k}+\frac{1}{2},(\frac{b}{\tau k}(\frac{\rho_{2}}{b}-\frac{\rho_{1}}{k}))^{2}(\frac{\beta^{2}}{4\varepsilon_{1}}+\frac{\tau^{2}}{2}),\frac{c^{2}\tau^{2}}{4k^{2}\varepsilon_{1}})$ and $\varepsilon_{1}>0$.
\end{lemma}
\begin{proof}
 By differentiation of \eqref{k4}, using \eqref{1} and integration over $(0,1)$, we get
 \begin{eqnarray*}
K_{4}'(t)=&&\! \! \! \! \! \! \! \! \frac{\tau}{2}\int^{1}_{0}(b\psi_{xx}-k(\varphi_{x}+\psi)-\delta\theta_{x}-\alpha(t)h(\psi_{t}))(\varphi_{x}+\psi)dx\\&+&\frac{\tau\rho_{2}}{k}\int^{1}_{0}\psi_{t}(\varphi_{x}+\psi)_{t}dx+\frac{b\tau}{k^{2}}\int^{1}_{0}(\varphi_{x}+\psi )_{x}\varphi_{x}+\varphi_{t}\psi_{tx}dx\\
&-&\frac{b\tau}{\delta k}(\frac{\rho_{2}}{b}-\frac{\rho_{1}}{k})\int^{1}_{0}(-(q_{x}+\delta \psi _{xt})\varphi_{t}+\theta (\varphi_{x}+\psi )_{x}) dx\\
&+&\frac{b}{\delta k}(\frac{\rho_{2}}{b}-\frac{\rho_{1}}{k})\int^{1}_{0}(-(\beta q+\theta _{x})(\varphi_{x}+\psi )+q (\varphi_{x}+\psi )_{t}) dx.
 \end{eqnarray*}
 By integration over $(0,1)$ and using the boundary conditions $\eqref{1}_{5}$, we have
\begin{eqnarray*}
 K_{4}'(t)=&-&\tau \int_{0}^{1}(\varphi_{x}+\psi )^{2}dx+\frac{\tau\rho_{2}}{k}\int_{0}^{1}\psi_{t}^{2}dx
 +\frac{b\tau}{\delta k}(\frac{\rho_{2}}{b}-\frac{\rho_{1}}{k})\int^{1}_{0}q\psi_{t} dx\\&-&\frac{b\beta}{\delta k}(\frac{\rho_{2}}{b}
 -\frac{\rho_{1}}{k})\int^{1}_{0}q(\varphi_{x}+\psi ) dx-\frac{\tau}{k}\int_{0}^{1}\alpha(t)h(\psi_{t})(\varphi_{x}+\psi)dx\\
 &+&\frac{b\rho_{3}}{\delta \rho_{1}}\left[ (\tau -\frac{\rho_{1}}{k \rho_{3}})(\frac{\rho_{2}}{b}-\frac{\rho_{1}}{k})-\frac{\tau\delta^{2}\rho_{1}}{bk \rho_{3}}\right] \int^{1}_{0}\theta_{x}(\varphi_{x}+\psi ) dx.
 \end{eqnarray*}
 Applying Young's inequality, we obtain \eqref{k4-prime}.
\end{proof}
Next, we define a Lyapunov functional $K$ and show that it is equivalent to the
energy functional $E$.
\begin{lemma}
Let $(\varphi,\psi,\theta,q)$ be a solution of the system \eqref{1}. Then, the functional
\begin{equation}\label{fonctionlyapo}
K(t):=NE(t)+K_{1}+N_{2}K_{2}+N_{3}K_{3}+N_{4}K_{4},
\end{equation}
where $N$ is sufficiently large, $N_{1}$ and $N_{2}$ are positive real numbers to be chosen properly, satisfies
\begin{equation}
c_{1}E(t)\leq K(t)\leq c_{2}E(t),\ \label{equation equivalen}
\end{equation}
for $c_{1}$ and $c_{2}$ two positive constants and
\begin{eqnarray}
K'(t)\leq \! \! \! \!&-&\! \! \! \!(\rho_{1}-N_{2}\rho_{1}\varepsilon )\int_{0}^{1}\! \! \varphi_{t}^{2}dx-\rho_{2}\int_{0}^{1}\! \!\psi _{t}^{2}dx-(N_{2}(b-2c\varepsilon )-c) \int_{0}^{1}\! \!\psi _{x}^{2}dx\label{k'}\\
&-&\int_{0}^{1}(N_{4}(\tau -2\varepsilon_{1})-k))(\varphi_{x}+\psi) ^{2}dx-(\frac{N_{3}\rho_{3}}{2}-\frac{\delta}{2})\int_{0}^{1}\theta^{2}dx\nonumber\\
&-&(N\beta -cN_{2}-cN_{3}-cN_{4}) \int_{0}^{1}q^{2}dx+c\int_{0}^{1}(\psi _{t}^{2}+h^2(\psi_{t}))dx\nonumber\\
&+&N_{4}\frac{b\rho_{3}}{\delta \rho_{1}}\left[ (\tau -\frac{\rho_{1}}{k \rho_{3}})(\frac{\rho_{2}}{b}-\frac{\rho_{1}}{k})-\frac{\tau\delta^{2}\rho_{1}}{bk \rho_{3}}\right]\int^{1}_{0}\theta_{x}(\varphi_{x}+\psi ) dx.\nonumber
\end{eqnarray}
\end{lemma}
\begin{proof}
From Lemmas \ref{lemmak1} to \ref{lemmak4}, we find
\begin{eqnarray*}
\vert K(t)-NE(t)\vert \leq \! \! \! \!&&\! \! \! \! \! \!\rho_{1}\int^{1}_{0}\vert \varphi \varphi_{t}\vert dx+(\rho_{2}+N_{2})\int^{1}_{0}\vert \psi\psi_{t}\vert dx+N_{2}\rho_{1}\int^{1}_{0}\vert \varphi_{t}w\vert dx\\&+&N_{2}\tau\delta\int^{1}_{0}\vert \psi q\vert dx+\tau\rho_{3}\int^{1}_{0}\vert q(\int^{x}_{0}\theta(t,y)dy)\vert dx.
\end{eqnarray*}
Applying Young, Poincar\'e and Cauchy-Schwartz inequalities and the fact that
$$\varphi_{x}^{2}\leq 2(\varphi_{x}+\psi)^{2}+2\psi^{2}\leq 2(\varphi_{x}+\psi)^{2}+2c\psi_{x}^{2},$$
we obtain \eqref{equation equivalen}, and therefore we get
$$K(t)\sim E(t).$$
For to prove \eqref{k'}, it suffices to differentiate \eqref{fonctionlyapo} and use lemmas \ref{lemmaenergie}-\ref{lemmak4}.
This ends the proof of the lemma.
\end{proof}
\begin{theorem}
\label{Theop} Let us suppose that
$$\mu =\left[ (\tau -\frac{\rho_{1}}{k \rho_{3}})(\frac{\rho_{2}}{b}-\frac{\rho_{1}}{k})-\frac{\tau\delta^{2}\rho_{1}}{bk \rho_{3}}\right]=0.$$
Then there exist positive constants $k_{1}$, $k_{2}$, $k_{3}$ and $\varepsilon_{0}$ such that the energy $E(t)$ associated with \eqref{1} satisfies
\begin{equation}
\left. E(t)\leq k_{3} H^{-1}_{1} \left( k_{1} \int_{0}^{t} \alpha(s)\ ds+
k_{2} \right), \hspace{1.cm} \text{for all} \ \ t\geq 0, \right.  \label{theoremenergy1}
\end{equation}
where
\[
H_{1}(t) = \int_{t}^{1} \frac{1}{H_{2}(s)}ds,\hspace{1cm} H_{2}(t)=tH^{^{%
\prime }}(\varepsilon_{0}t).
\]
Here $H_{1}$ is a strictly decreasing and convex function on $(0,1]$, with $\displaystyle\lim_{t\rightarrow0}H_{1}(t)= +\infty$.
\end{theorem}
\begin{proof}
The estimate \eqref{k'}, with $\mu=0$, takes the form
\begin{eqnarray}
K'(t)\leq \! \! \! &-&(\rho_{1}-N_{2}\rho_{1}\varepsilon )\int_{0}^{1} \varphi_{t}^{2}dx-\rho_{2}\int_{0}^{1}\psi _{t}^{2}dx-(N_{2}(b-2c\varepsilon )-c) \int_{0}^{1}\psi _{x}^{2}dx\nonumber\\
&-&\int_{0}^{1}(N_{4}(\tau -2\varepsilon_{1})-k))(\varphi_{x}+\psi) ^{2}dx-(\frac{N_{3}\rho_{3}}{2}-\frac{\delta}{2})\int_{0}^{1}\theta^{2}dx\nonumber\\
&-&(N\beta -cN_{2}-cN_{3}-cN_{4}) \int_{0}^{1}q^{2}dx+c\int_{0}^{1}(\psi _{t}^{2}+h^2(\psi_{t}))dx.\nonumber
\end{eqnarray}
Now, we choose the constants in the above estimate as follows: first  $\varepsilon$ and $\varepsilon_{1}$ are such that
$$\varepsilon =\frac{1}{2N_{2}}\hspace{0.5cm}\text{and}\hspace{0.5cm}\varepsilon_{1}<\frac{\tau}{2}.$$
After that, we choose $N$, $N_{2}$, $N_{3}$ and $N_{4}$ sufficiently large such that $N_{2}>\frac{2c}{b}$, $N_{3}>\frac{\delta}{\rho_{3}}$, $N_{4}>\frac{k}{\tau -2\varepsilon_{1}}$ and $N>\frac{c}{\beta}(\frac{2c}{b}+\frac{\delta}{\rho_{3}}+\frac{k}{\tau -2\varepsilon_{1}})$.
Then, we deduce that
\begin{equation}
K^{\prime }(t)\leq -dE(t)+c\int^{1}_{0}(\psi_{t}^{2}+h^{2}(\psi_{t})) dx, \label{k'energ}
\end{equation}
where $d=\min(\rho_{1}-N_{2}\rho_{1}\varepsilon ,\rho_{2},N_{2}(b-2c\varepsilon )-c,N_{4}(\tau -2\varepsilon_{1})-k,\frac{N_{3}\rho_{3}}{2}-\frac{\delta}{2},N\beta -cN_{2}-cN_{3}-cN_{4}).$
\begin{itemize}
\item[\protect\underline{First case:}] Let $h_{0}$ be a linear function over $%
[0,\varepsilon ]$. The hypothesis $(A_{2})^{*}$ implies that
\[
c_{1}^{\prime }|s|\leq |h(s)|\leq c_{2}^{\prime }|s|,\hspace{1cm}\text{for all}
\ s\in \mathbb{R}.
\]%
Consequently, by multiplying inequality \eqref{k'energ} by $\alpha(t)$, we obtain
\begin{eqnarray}
\alpha (t)K^{\prime }(t)\! \! \! &\leq &\! \! \! -d\alpha (t)E(t)+c\ \alpha
(t)\int_{0}^{1}(\psi _{t}^{2}+h^{2}(\psi _{t}))dx,  \label{312} \\
&\leq &-d\alpha (t)E(t)+c\alpha (t)\int_{0}^{1}(\frac{1}{c_{1}^{\prime }}%
|\psi _{t}h(\psi _{t})|+c_{2}^{\prime }|\psi _{t}h(\psi _{t})|)dx, \nonumber \\
&\leq &-d\alpha (t)E(t)+c_{0}\alpha (t)\int_{0}^{1}\psi _{t}h(\psi
_{t})dx=-d\alpha (t)E(t)-c_{0}E^{\prime }(t),  \nonumber
\end{eqnarray}%
where $c_{0}=c(\frac{1}{c_{1}^{\prime }}+c_{2}^{\prime })$.\newline
Using now hypotesis $(A_{1})$, this yields
\begin{equation}
(\alpha K+c_{0}E)^{\prime }(t)\leq \alpha (t)K^{\prime
}(t)+c_{0}E^{\prime }(t)\leq -d\alpha (t)E(t).  \label{maj}
\end{equation}%
We integrate the inequality \eqref{maj} and use the fact that $\alpha K+c_{0}E\sim E$,  we obtain for some $k,\ c>0$,
\begin{equation}
E(t)\leq k\exp(-dc\int_{0}^{t}\alpha(s)ds). \label{eqfinal}
\end{equation}
Finally, by a simple computation we get \eqref{theoremenergy1}.
\end{itemize}

\begin{itemize}
\item[\protect\underline{Second case:}] Let $h_{0}$ be a non-linear function over $[0,\varepsilon ]$. We assume that $\max (r,h_{0}(r))<\varepsilon$, where $r$ is defined in the hypothesis $(A_{2})^{*}$.\\
  Let $\varepsilon _{1}=\min (r,h_{0}(r))$, we deduce from the hypothesis $ (A_ {2})^{*} $ that
\[
\frac{h_{0}(\varepsilon _{1})}{\varepsilon }|s|\leq \frac{h_{0}(|s|)}{|s|}%
|s|\leq |h(s)|\leq \frac{h_{0}^{-1}(|s|)}{|s|}|s|\leq \frac{%
h_{0}(\varepsilon )}{\varepsilon _{1}}|s|,
\]%
for all $s$ satisfying $\DST \varepsilon_{1}\leq |s|\leq \varepsilon $.\\
Then, the estimates in hypothesis $(A_{2})^{*}$ become
\begin{equation}
\left\{
\begin{array}{l}
h_{0}(|s|)\leq |h(s)|\leq h_{0}^{-1}(|s|),\hspace{1.7cm}\text{for all}\hspace{%
1em}|s|\leq \varepsilon _{1} ,\\
c_{1}^{^{\prime }}|s|\leq |h(s)|\leq c_{2}^{^{\prime }}|s|,\hspace{6.9em}%
\text{for all}\hspace{0.9em}|s|\geq \varepsilon _{1},%
\end{array}%
\right.   \label{313}
\end{equation}
and we have

\begin{equation}
s^{2}+h^{2}(s)\leq 2H^{-1}(sh(s)) .\label{est3}
\end{equation}
To estimate the last term of (\ref{k'energ}), we consider the following partition of $(0.1)$:
\[
\Omega_{1}=\{x\in (0,1); |\psi_{t}|\leq\varepsilon_{1}\},\
\Omega_{2}=\{x\in (0,1); |\psi_{t}|>\varepsilon_{1}\} .
\]
Then, we obtain
\begin{equation}
\left. \psi_{t}h(\psi_{t})\leq H(r^{2}) \ \ \text{and}\ \ \psi_{t}h(\psi_{t})\leq
r^{2} \hspace{0.5cm}\text{on}\ \Omega_{1} .\right.  \label{315}
\end{equation}
Now, we apply Jensen's inequality to the following term
\[
I(t):=\frac{1}{|\Omega_{1}|}\int_{\Omega_{1}}\psi_{t}h(\psi_{t}) dx,
\]%
and we infer that
\begin{equation}
H^{-1}(I(t))\geq c \int_{\Omega_{1}}H^{-1}(\psi_{t}h(\psi_{t})) dx. \label{est4}
\end{equation}
Using (\ref{313}), (\ref{est3}) and (\ref{est4}), then the right-hand side of \eqref{k'energ} multiplied by $\alpha(t)$ becomes
\begin{eqnarray*}
\alpha(t)\int_{0}^{1}(\psi_{t}^{2}+h^{2}(\psi_{t})) dx &=&
\alpha(t)\int_{\Omega_{1}}(\psi_{t}^{2}+h^{2}(\psi_{t})
)dx+\alpha(t)\int_{\Omega_{2}}(\psi_{t}^{2}+h^{2}(\psi_{t})) dx,\\
&\leq& 2\alpha(t) \int_{\Omega _{1}}H^{-1}(\psi_{t}h(\psi_{t})) dx\\
&+&\alpha(t)\int_{\Omega _{2}}(|\psi_{t}|\frac{1}{c^{\prime }_{1}}%
\vert h(\psi_{t})\vert+c^{\prime }_{2}|\psi_{t}||h(\psi_{t})|) dx ,\\
&\leq& c\alpha(t)H^{-1}(I(t))
+\alpha(t)c\int_{0}^{1}\psi_{t}h(\psi_{t}) dx, \\
&\leq& c\alpha(t)H^{-1}(I(t))-cE^{\prime }(t).
\end{eqnarray*}
 Consequently, the estimate \eqref{k'energ} gives
\begin{equation}
\left. R^{\prime }_{0}(t) \leq -d\alpha (t) E(t)+c\alpha
(t)H^{-1}(I(t))\right.  ,\label{318}
\end{equation}
where $R_{0}=\alpha K+cE$. \\
On the one hand, for $\varepsilon_{0}<r^{2}$, using (\ref{318}), $H^{\prime }\geq0$ and $H^{''}\geq0 $ over $(0,r^{2}]$ and $E^{\prime }\leq0$ the functional $R_{1}$ defined by
\[
R_{1}(t):=H^{\prime }\bigg(\varepsilon_{0}\frac{E(t)}{E(0)}\bigg)R_{0}(t)+c_{0}E(t),
\]
is equivalent to $E(t).$ \\
On the other hand, using the fact that $\varepsilon_{0}\frac{E^{^{\prime }}(t)}{E(0)}H^{\prime \prime
}(\varepsilon_{0}\frac{E(t)}{E(0)})R_{0}(t)\leq 0$ and (\ref{318}), we conclude that
\begin{eqnarray}
R^{\prime }_{1}(t)\! \! \!&=&\! \! \!\varepsilon_{0}\frac{E^{^{\prime }}(t)}{E(0)}%
H^{''}(\varepsilon_{0}\frac{E(t)}{E(0)})R_{0}(t)+H^{\prime
}(\varepsilon_{0}\frac{E(t)}{E(0)})R_{0}^{'}(t)+c_{0}E^{\prime }(t)\label{r1'}\\
&\leq& -d\alpha (t)E(t)H^{\prime }(\varepsilon_{0}\frac{E(t)}{E(0)})+c\alpha(t)H^{^{\prime }}(\varepsilon_{0}\frac{E(t)}{E(0)})H^{-1}(I(t))+c_{0}E^{\prime}(t).\nonumber
\end{eqnarray}

Our goal now is to estimate the second term in the right-hand side of \eqref{r1'}. For that purpose, we introduce the convex conjugate $H^{*}$ of $H$ defined by
\begin{equation}\label{conjugate}
H^{\ast}(s)=s(H^{\prime })^{-1}(s)-H((H^{\prime })^{-1}(s)\text{ for } s\in(0,H^{\prime }(r^{2})),
\end{equation}
and $ H ^ {*} $  satisfies the following Young inequality:
\begin{equation}\label{ableq}
AB\leq H^{\ast}(A)+H(B) \text{ for } A\in (0,H^{^{\prime }}(r^{2})),\ B\in (0,r^{2}).
\end{equation}
Now, taking  $A=H^{\prime }(\varepsilon_{0}\frac{E(t)}{E(0)})$ and $%
B=H^{-1}(I(t))$, we obtain
\begin{eqnarray}
R^{\prime }_{1}(t)&\leq& -d\alpha (t)E(t)H^{\prime }(\varepsilon_{0}%
\frac{E(t)}{E(0)})+c\alpha(t)H^{\ast}\left( H^{\prime }(\varepsilon_{0}%
\frac{E(t)}{E(0)})\right) \nonumber \\
&&+c \alpha(t)H\left( H^{-1}(I(t)\right)+c_{0}E^{\prime }(t)\nonumber\\
&\leq &-d\alpha (t)E(t)H^{\prime }(\varepsilon_{0}%
\frac{E(t)}{E(0)})+c\varepsilon_{0}\frac{E(t)}{E(0)}\alpha(t) H^{\prime }(\varepsilon_{0}\frac{E(t)}{E(0)})\nonumber\\
&&- c\alpha(t) H(\varepsilon_{0}\frac{E(t)}{E(0)})+c \alpha(t)I(t) +c_{0}E^{\prime }(t) \nonumber\\
&\leq &-d\alpha (t)E(t)H^{\prime }(\varepsilon_{0}%
\frac{E(t)}{E(0)})+c\varepsilon_{0}\frac{E(t)}{E(0)}\alpha(t) H^{\prime}(\varepsilon_{0}\frac{E(t)}{E(0)})- cE^{\prime }(t)+c_{0}E^{\prime }(t).\nonumber\label{r1'2}\hspace{2cm}
\end{eqnarray}
With a suitable choice of $ \varepsilon_{0} $ and $c_ {0}$, we deduce from the last inequality that
\begin{equation}
R^{\prime }_{1}(t)\leq -(dE(0)-c\varepsilon_{0})\alpha (t)\frac{E(t)}{E(0)%
}H^{\prime }(\varepsilon_{0}\frac{E(t)}{E(0)})\leq -k\alpha(t)H_{2}(\frac{%
E(t)}{E(0)}),\label{r1'3}
\end{equation}
where $k=dE(0)-c\varepsilon_{0}>0$ and $H_{2}(s)=sH^{'}(\varepsilon_{0}s)$.\\
	
Since $E(t)\sim R_{1}(t)$, then there exist $a_{1}$ and $a_{2}$ such that
$$a_{1}R_{1}(t)\leq E(t)\leq a_{2}R_{1}(t).$$\label{319}
We set now $R(t)=\frac{a_{1}R_{1}(t)}{E(0)}$. It is clear that  $R(t)\sim E(t)$. We use the fact that $H_{2}^{'}(t),\ H_{2}(t)>0$ over $(0,1]$ (this is due to the fact that $H$ is strictly convex on $(0,r^{2}]$) and we deduce from (\ref{r1'3}) that
\[
R^{\prime }(t)\leq-k_{1}\alpha(t)H_{2}(R(t)),\hspace{0,5cm}\text{ for all }t\in \mathbb{R_{+}}, \]
with  $k_{1}  >0$.\\
By integrating the last inequality, we obtain

$$
H_{1}(R(t))\geq H_{1}(R(0))+k_{1}\int^{t}_{0}\alpha(s)\ ds.
$$
Finally, using the fact that $H_{1}^{-1}$ is decreasing (because $H_{1}$ is also), we have
\[
R(t)\leq H^{-1}_{1} \left( k_{1} \int_{0}^{t} \alpha(s)\ ds+
k_{2} \right), \hspace{1cm}\text{with} \ k_{2}>0.
\]
Taking into account that $E(t)\sim R(t),$ we deduce (\ref{theoremenergy1}).
\end{itemize}
\end{proof}
\subsubsection{Examples}
In the following, we will apply the inequality \eqref{theoremenergy1} on some examples in order to show explicit stability results in term of asymptotic profiles in time. For that, we choose the function $H$ strictly convex near zero.
\begin{itemize}
\item[\underline{Example 1.}] \ \\
Let $h$ be a function that satisfies
\[
c_{3}\min(\vert s\vert,\vert s\vert^{p})\leq \vert h(s)\vert \leq c_{4}\max
(\vert s\vert,\vert s\vert^{\frac{1}{p}}),
\]
with some $c_{3}$, $c_{4}>0$ and $p\geq1$. \\
For $h_{0}(s)=cs^{p} $, hypothesis $(A_{2})^{*}$ is verified.
Then $H(s)=cs^{\frac{p+1}{2}}$. \vspace{0.2cm}\\
Therefore, we distinguish the following two cases:\\
$\bullet$ If p=1, we have $h_{0}$ is linear and  $H_{2}(s)=cs $, $H_{1}(s)=-%
\frac{\ln(s)}{c}$ and $H_{1}^{-1}(t)=\exp(-ct)$.\newline
Applying (\ref{theoremenergy1}) of Theorem \ref{Theop}, we conclude that
\[
E(t)\leq k_{3}\exp(-c(k_{1}\int^{t}_{0}\alpha(s)\
ds+k_{2})).
\]

$\bullet$ If $p>1$; this implies that $h_{0}$ is nonlinear and we have $H_{2}(s)=c\frac{p+1}{2}%
\varepsilon_{0}^{\frac{p-1}{2}}s^{\frac{p-1}{2}} $
 and
\[
H_{1}(t)=\int_{t}^{1}\frac{1}{\delta}s^{-\frac{p-1}{2}}\ ds =\frac{2}{\delta
(1-p)}-\frac{2}{\delta (1-p)}t^{\frac{p-1}{2}},\hspace{1cm}\text{with }\delta =c\frac{p+1}{2}\varepsilon_{0}^{\frac{p-1}{2}}
.\]
Therefore,
\[
H_{1}^{-1}(t)=(\delta\frac{p-1}{2}t+1)^{-\frac{2}{p-1}} .
\]
Using again (\ref{theoremenergy1}), we obtain
\[
E(t)\leq H_{1}^{-1}( k_{1}\int_{0}^{t}\alpha(s)\ ds +k_{2})=(\delta\frac{p-1%
}{2}( k_{1}\int_{0}^{t}\alpha(s)\ ds +k_{2})+1)^{-\frac{2}{p-1}} .
\]
\item[\underline{Example 2.}] \ \\
 Let $h_{0}(s)=\exp(-\frac{1}{s})$, this yields $H(s)=\sqrt{s}\exp(-%
\frac{1}{\sqrt{s}})$ and
\[
H_{2}(s)=(\frac{\sqrt{s}}{2\sqrt{\varepsilon}_{0}}+\frac{1}{2\varepsilon_{0}}%
)\exp(-\frac{1}{\sqrt{\varepsilon_{0}s}}) .
\]
Moreover, we have
\begin{eqnarray*}
H_{1}(t)&=&\int^{1}_{t}\left( \frac{1}{\frac{\sqrt{s}}{2\sqrt{\varepsilon}_{0}}+\frac{1}{2\varepsilon_{0}}}\right) \exp(\frac{1}{\sqrt{\varepsilon_{0}s}})ds\\ &\leq &\int^{1}_{t}\frac{2\sqrt{\varepsilon}_{0}}{\sqrt{s}}\exp(\frac{1}{\sqrt{\varepsilon_{0}s}})ds\\
&\leq &c\int^{1}_{t}\frac{1}{2s\sqrt{\varepsilon_{0}s}}\exp(\frac{1}{\sqrt{\varepsilon_{0}s}})ds=c\exp(\frac{1}{\sqrt{\varepsilon_{0}t}})-c\exp(\frac{1}{\sqrt{\varepsilon_{0}}}).
\end{eqnarray*}
Then,
\[
t\leq \varepsilon_{0}^{-1}\left( \ln\left( \frac{H_{1}(t)+c\exp(\frac{1}{\sqrt{\varepsilon_{0}s}})}{c}\right) \right)^{-2} .
\]
Replacing $t$ by $\displaystyle H_{1}^{-1}\left(
k_{1}\int_{0}^{t}\alpha(s)\ ds +k_{2}\right) $ in the last inequality, we find 
\[
H_{1}^{-1}\left( k_{1}\DST \int_{0}^{t}\alpha(s)\ ds +k_{2}\right) \leq
 \varepsilon_{0}^{-1}\left(  \ln\left( \frac{k_{1}\DST \int_{0}^{t}\alpha(s)\ ds +k_{2}+c\exp(\frac{1}{\sqrt{\varepsilon_{0}}})}{c}\right) \right)^{-2}.
\]
Therefore,
 $$E(t)\leq k_{3}\varepsilon_{0}^{-1}\left(  \ln\left( \frac{k_{1}\DST \int_{0}^{t}\alpha(s)\ ds +k_{2}+c\exp(\frac{1}{\sqrt{\varepsilon_{0}}})}{c}\right) \right)^{-2}.$$%

\item[\underline{Example 3.}] \ \\
Let $h_{0}(s)=\frac{1}{s}\exp(-\frac{1}{s^{2}})$. Following the same steps in exemple 2 we find that the energy of (\ref{1}) satisfies
$$E(t)\leq \varepsilon \left( \ln(\frac{k_{1}\int_{0}^{t}\alpha(s)\ ds +k_{2}+c\exp(\frac{1}{\varepsilon_{0}})}{c})\right)^{-1}.
$$
\item[\underline{Example 4.}] \ \\
Let  $h_{0}(s)=\frac{1}{s}\exp (-\frac{1}{4}(\ln s)^{2})$. Then, we have $H(s)=\exp (-%
\frac{1}{4}(\ln s)^{2}),$ \\ $%
H_{2}(s)=-\frac{1}{2}\frac{\ln \varepsilon_{0}s}{\varepsilon_{0}}\exp (-%
\frac{1}{4}(\ln \varepsilon_{0}s)^{2})$ and
$H_{1}(t)=\int_{t}^{1}-2\frac{\varepsilon_{0}}{\ln \varepsilon_{0}s}\exp (%
\frac{1}{4}(\ln \varepsilon_{0}s)^{2}).$ \\
As $\DST \lim_{s\longrightarrow0} \frac{4\varepsilon_{0}^{2}s}{(\ln(\varepsilon_{0}s))^{2}}=0$, then the function $s\mapsto \frac{4\varepsilon_{0}^{2}s}{(\ln(\varepsilon_{0}s))^{2}}$ is bounded on $(0,1]$, and we infer that 

\[
H_{1}(t)\leq c\int_{t}^{1}-\frac{1}{2}\frac{\ln \varepsilon_{0}s}{%
\varepsilon_{0}s}\exp (\frac{1}{4}(\ln s)^{2})\ ds=\exp (\frac{1}{4}(\ln
\varepsilon_{0}t)^{2})-\underbrace{\exp (\frac{1}{4}(\ln
\varepsilon_{0})^{2})}_{c_{1}}.
\]
Hence, we have
\[
t\leq \frac{1}{\varepsilon_{0}}\exp\left( -2\left( \ln
(H_{1}(t))+c_{1}\right)^{\frac{1}{2}}\right).
\]
Replacing $t$ by $\displaystyle H_{1}^{-1}\left( k_{1}\int_{0}^{t}\alpha(s)\
ds +k_{2}\right) $ in the last inequality, we find
\[
E(t)\leq k_{3}H_{1}^{-1}\left( k_{1}\int_{0}^{t}\alpha(s)\ ds +k_{2}\right)=
\frac{k_{3}}{\varepsilon_{0}}\exp\left( -2\left( \ln k_{1}\int_{0}^{t}\alpha(s)\
ds +k_{2}+c_{1}\right)^{\frac{1}{2}}\right).
\]
\end{itemize}

\subsection{The case $\mu \neq 0$  and  $\alpha(t)=1$.} \ \\
This section is devoted to the statement and the proof of the stability result for the system \eqref{1} when $\mu \neq 0$  and  $\alpha(t)=1$. \\ We have the following theorem.
\begin{theorem}\label{th3}
Let us suppose that conditions $(A_{1})$ and $(A_{2})^{*}$ hold, then for
$$\mu =\left[ (\tau -\frac{\rho_{1}}{k \rho_{3}})(\frac{\rho_{2}}{b}-\frac{\rho_{1}}{k})-\frac{\tau\delta^{2}\rho_{1}}{bk \rho_{3}}\right]\neq 0,$$
the energy solution of \eqref{1} satisfies

\begin{equation}
 E(t)\leq H_{2}^{-1}(\frac{c}{t}),\label{theoremenergy2}
 \end{equation}
where
\[
H_{2}(t)=tH^{^{\prime }}(\varepsilon_{0}t) \mbox{ with } \displaystyle\lim_{t\rightarrow0}H_{2}(t)=0.
\]

\end{theorem}

\begin{proof}
Let $(\varphi,\psi,\theta,q)$ be a solution of the system \eqref{1}. First, we define
\begin{equation}
E(t):=\frac{1}{2}\int_{0}^{1}\left( \rho_{1} \varphi_{t}^{2}+\rho_{2}\psi _{t}^{2}+b \psi _{x}^{2}+k(\varphi_{x}+\psi) ^{2}+\rho_{3}\theta^{2}+\tau q^{2}\right) dx,
\end{equation}
and
\begin{equation}
\tilde{E}(t):=\frac{1}{2}\int_{0}^{1}\left( \rho_{1} \varphi_{tt}^{2}+\rho_{2}\psi _{tt}^{2}+b \psi _{tx}^{2}+k(\varphi_{tx}+\psi_{t}) ^{2}+\rho_{3}\theta_{t}^{2}+\tau q_{t}^{2}\right) dx.
\end{equation}
Then, the functional $E$ satisfies
\begin{equation}
E^{\prime
}(t)=-\beta \int^{1}_{0}q^{2}dx-\int^{1}_{0}\psi_{t}h(\psi_{t}) dx \leq 0.
\end{equation}
Analogously, the functional $\tilde{E}$ satisfies
\begin{equation}
\tilde{E}^{\prime
}(t)=-\beta \int^{1}_{0}q_{t}^{2}dx-\int^{1}_{0}\psi_{tt}^{2}h'(\psi_{t}) dx \leq 0.
\end{equation}
\\
Using the results in Subsection \ref{sub31} (recall the expressions of the functionals $K_{1},...,K_{4}$) we have the following Lemma.
\begin{lemma}
Let $(\varphi,\psi,\theta,q)$ be a solution of the system \eqref{1}. Then, the functional

\begin{equation}\label{L(t)}
L(t):=N(E(t)+\tilde{E}(t))+K_{1}+N_{2}K_{2}+N_{3}K_{3}+N_{4}K_{4},
\end{equation}
satisfies
\begin{eqnarray}\label{l'}
L'(t)&\leq& -d'(t)+c\int_{0}^{1}(\psi _{t}^{2}+h^2(\psi_{t}))dx,
\end{eqnarray}
for $N$ large enough and $d^{'}>0$.
\end{lemma}
\begin{proof}
 By differentiation of \eqref{L(t)}, and using \eqref{k'} and Young's inequality, we obtain
 \begin{eqnarray}
 L'(t)\leq -dE(t)+c\int_{0}^{1}(\psi _{t}^{2}+h^2(\psi_{t}))dx+c\int_{0}^{1}(\theta_{x}^{2}+(\varphi_{x}+\psi)^{2})dx\\
 -N\beta \int^{1}_{0}q_{t}^{2}dx-N\int^{1}_{0}\psi_{tt}^{2}h'(\psi_{t}) dx.\nonumber
\end{eqnarray}
Now, from $(\ref{1})_{4}$, we deduce that
 $$\int_{0}^{1}\theta_{x}^{2}dx\leq c\left( \int_{0}^{1}q^{2}dx+\int_{0}^{1}q_{t}^{2}dx\right).$$
 Consequently, we get
 \begin{eqnarray}
 L'(t)\leq -d'E(t)+c\int_{0}^{1}(\psi _{t}^{2}+h^2(\psi_{t}))dx-(\beta N-c)\int_{0}^{1}q_{t}^{2}dx\\
- N\int^{1}_{0}\psi_{tt}^{2}h'(\psi_{t}) dx.\hspace{4cm}\nonumber
 \end{eqnarray}
 where $d'=d-c>0$ and $d$ is the same constant that appears in \eqref{k'energ}. Finally, we choose $N$ large enough and using the monotonie of the function $h$ we arrive at (\ref{l'}).
\end{proof}
Now, using the following partion of $(0,1)$ defined in Subsection \ref{sub31}, the right-hand side of \eqref{l'} becomes
\begin{eqnarray*}
\int_{0}^{1}(\psi_{t}^{2}+h^{2}(\psi_{t})) dx =
\int_{\Omega_{1}}(\psi_{t}^{2}+h^{2}(\psi_{t})
)dx+\int_{\Omega_{2}}(\psi_{t}^{2}+h^{2}(\psi_{t})) dx.
\end{eqnarray*}
Now, the estimates \eqref{313}-\eqref{est4} imply that
\begin{eqnarray*}
\int_{0}^{1}(\psi_{t}^{2}+h^{2}(\psi_{t})) dx &\leq& 2 \int_{\Omega _{1}}H^{-1}(\psi_{t}h(\psi_{t})) dx\\
&&+\int_{\Omega _{2}}(|\psi_{t}|\frac{1}{c^{\prime }_{1}}%
\vert h(\psi_{t})\vert+c^{\prime }_{2}|\psi_{t}||h(\psi_{t})|) dx \\
&\leq& cH^{-1}(I(t))
+c\int_{0}^{1}\psi_{t}h(\psi_{t}) dx.\\
%
\end{eqnarray*}
 Consequently,
\begin{eqnarray*}
L'(t)&\leq& -d'E(t)+cH^{-1}(I(t))
+c\int_{0}^{1}\psi_{t}h(\psi_{t}) dx+c\beta \int_{0}^{1}q^{2}dx\\
&\leq& -d'E(t)+cH^{-1}(I(t))-cE'(t).
\end{eqnarray*}
Hence, we deduce that
\begin{equation}
\left. (L+cE)^{\prime }(t) \leq -d' E(t)+cH^{-1}(I(t)).\right.
\end{equation}
We then define
\[
R_{1}(t):=H^{\prime }(\varepsilon_{0}\frac{E(t)}{E(0)})(L+cE)(t)+c_{0}E(t),
\]
which verifies
\begin{eqnarray}
R^{\prime }_{1}(t)
\leq -d_{1}E(t)H^{\prime }(\varepsilon_{0}\frac{E(t)}{E(0)})+cH^{^{\prime }}(\varepsilon_{0}\frac{E(t)}{E(0)})H^{-1}(I(t))+\epsilon E^{\prime}(t),
\end{eqnarray}
as we have  $\varepsilon_{0}\frac{E^{^{\prime }}(t)}{E(0)}H^{\prime \prime
}(\varepsilon_{0}\frac{E(t)}{E(0)})R_{0}(t)\leq 0$.
\\
We recall the definition of the convex conjugate $ H ^ {*} $ of $H$, given by (\ref{conjugate}), which satisfies the following Young inequality:
\[
AB\leq H^{\ast}(A)+H(B) \text{ for } A\in (0,H^{^{\prime }}(r^{2})),\ B\in (0,r^{2}).
\]
With the same choice of $A$ and $B$ as in (\ref{ableq}), we obtain
\begin{eqnarray}
R^{\prime }_{1}(t)\leq -d_{1} E(t)H^{\prime }(\varepsilon_{0}%
\frac{E(t)}{E(0)})+c\varepsilon_{0}\frac{E(t)}{E(0)}\alpha(t) H^{\prime}(\varepsilon_{0}\frac{E(t)}{E(0)})- cE^{\prime }(t)+\epsilon E^{\prime }(t)\nonumber.
\end{eqnarray}
With a suitable choice of $ \varepsilon_{0} $ and $\epsilon $, we deduce from the above inequality that
\begin{equation}
R^{\prime }_{1}(t)\leq -(dE(0)-c\varepsilon_{0}) \frac{E(t)}{E(0)%
}H^{\prime }(\varepsilon_{0}\frac{E(t)}{E(0)})\leq -k\alpha(t)H_{2}(\frac{%
E(t)}{E(0)}),\label{r1'3}
\end{equation}
where $k=dE(0)-c\varepsilon_{0}>0$ and $H_{2}(s)=sH^{'}(\varepsilon_{0}(s))$.\\
	
Finally, we have
\[
R_{1}^{\prime }(t)\leq-k_{1}H_{2}(\frac{E(t)}{E(0)}),\hspace{0,5cm}\text{ for all }t\in \mathbb{R_{+}}, \]
with  $k_{1}  >0$, which yields
$$\DST tH_{2}(\frac{E(t)}{E(0)})\leq \int^{t}_{0}H_{2}(\frac{E(s)}{E(0)})ds\leq -(R_{1}(t)-R_{1}(0))\leq R_{1}(0).$$
Then, we easily deduce that
$$\DST H_{2}(\frac{E(t)}{E(0)})\leq \frac{R_{1}(0)}{t}.$$
Thus,
$$ E(t)\leq E(0)H_{2}^{-1}(\frac{R_{1}(0)}{t}).$$
This concludes the proof of Theorem \ref{th3}.

\end{proof}
\subsubsection{Examples}
\ \\

\underline{Example 1:}
Let $h_{0}(s)=cs^{p} $. Then $H(s)=cs^{\frac{p+1}{2}}$. \vspace{0.2cm}\\
Therefore, we distinguish the following two cases:\\
$\bullet$ If p=1, we have $h_{0}$ is linear and $H_{2}^{-1}(t)=cs$.\newline
Applying (\ref{theoremenergy2}) of Theorem \ref{th3}, we conclude that
\[
E(t)\leq \frac{c}{t}.
\]
$\bullet$ If $p>1$; this implies that $h_{0}$ is nonlinear and we have $H_{2}(s)=cs^{\frac{p-1}{2}}.$
Therefore,
\[
H_{2}^{-1}(t)=ct^{\frac{2}{p-1}}.
\]
Using (\ref{theoremenergy2}), we obtain
\[
E(t)\leq ct^{-\frac{2}{p-1}}.
\]
\underline{Examples 2:}
 Let $h$ be given by $h(x)=\frac{1}{x^{3}}\exp(-\frac{1}{x^{2}})$ and we choose $h_{0}(x)=\frac{1+x^2}{x^3}\exp(-\frac{1}{x^{2}})$, we obtain $H(x)=\frac{1+x}{x}\exp(-%
\frac{1}{x})$ and
$ H_{2}(x)=\frac{\exp(-\frac{1}{\varepsilon_{0}x})}{\varepsilon_{0}^{3}x^2}.$ \\
Then, we use the following property :
$$\lim_{x\rightarrow 0^+}\exp(\frac{1}{\varepsilon_{0}x})H_{2}(x)=+\infty
,$$
and we deduce that $$\exp(-\frac{1}{\varepsilon_{0}x})\leq H_{2}(x).$$
We infer that there exists $x_0>0$ such that,
$$ \exp(-\frac{1}{\varepsilon_{0}x})\leq H_{2}(x) \mbox{ on } (0,x_0].$$
Consequently, the energy of the solution of \eqref{1} satisfies the estimate

$$E(t)\leq c(\ln(t))^{-1}.$$

\end{document}